\RequirePackage{fix-cm}
\documentclass[smallextended,numbook,envcountsame]{svjour3}       % onecolumn (second format)
\smartqed  % flush right qed marks, e.g. at end of proof
\usepackage{graphicx}
\usepackage{mathptmx}      % use Times fonts if available on your TeX system
\usepackage[cmex10]{amsmath}
\allowdisplaybreaks[2]
\usepackage{amsfonts}
\usepackage{bm}
\usepackage{psfrag}
\usepackage[tight,footnotesize]{subfigure}
\usepackage{verbatim}
\usepackage{cite}
\usepackage{url}
\providecommand{\abs}[1]{\left\lvert#1\right\rvert}
\providecommand{\norm}[2][]{\left\lVert#2\right\rVert_{#1}}

\DeclareMathOperator{\Diag}{Diag}

\DeclareMathOperator{\sgn}{sgn}

\DeclareMathOperator{\var}{var}

\newcommand{\0}{\mathbf{0}}
\newcommand{\mbb}{\mathbf{b}}
\newcommand{\mbQ}{\mathbf{Q}}
\newcommand{\mbf}{\mathbf{f}}
\newcommand{\mbc}{\mathbf{c}}

\newcommand{\mbI}{\mathbf{I}}

\newcommand{\mbx}{\mathbf{x}}
\newcommand{\mbi}{\mathbf{i}}
\newcommand{\mbe}{\mathbf{e}}
\newcommand{\mbmu}{\boldsymbol{\mu}}
\newcommand{\mbg}{\mathbf{g}}

\newcommand{\mbD}{\mathbf{D}}

\newcommand{\mbV}{\mathbf{V}}
\newcommand{\mbLambda}{\boldsymbol{\Lambda}}
\newcommand{\mbDelta}{\boldsymbol{\Delta}}
\newcommand{\mbU}{\mathbf{U}}
\newcommand{\mbPsi}{\boldsymbol{\Psi}}

\newcommand{\mbv}{\mathbf{v}}

\newcommand{\mbS}{\mathbf{S}}

\newcommand{\mbz}{\mathbf{z}}

\newcommand{\mbQt}{\widetilde{\mathbf{Q}}}

\newcommand{\mbct}{\widetilde{\mathbf{c}}}

\newcommand{\mbDeltat}{\widetilde{\boldsymbol{\Delta}}}
\newcommand{\mbUt}{\widetilde{\mathbf{U}}}

\newcommand{\lambdab}{\overline{\lambda}}

\newcommand{\mcE}{\mathcal{E}}

\newcommand{\eff}{\mathrm{eff}}

\newcommand{\dd}{\mathrm{dd}}
\newcommand{\na}{\mathrm{na}}

\spdefaulttheorem{assumption}{Assumption}{\bf}{\rm}
\begin{document}

\title{On two relaxations of quadratically-constrained cardinality minimization%\thanks{Grants or other notes
}


\author{Dennis Wei}


\institute{D. Wei \at
              Department of Electrical Engineering and Computer Science, University of Michigan, 1301 Beal Avenue, Ann Arbor, MI 48109, USA \\
%              Tel.: +123-45-678910\\
%              Fax: +123-45-678910\\
              \email{dlwei@eecs.umich.edu}           %  \\
}

\date{Received: \today / Accepted: date}

\maketitle

\begin{abstract}
This paper considers a quadratically-constrained cardinality minimization problem with applications to digital filter design, subset selection for linear regression, and portfolio selection.  Two relaxations are investigated: the continuous relaxation of a mixed integer formulation, and an optimized diagonal relaxation that exploits a simple special case of the problem.  For the continuous relaxation, an absolute upper bound on the optimal cost is derived, suggesting that the continuous relaxation tends to be a relatively poor approximation.  In computational experiments, diagonal relaxations often provide stronger bounds than continuous relaxations and can greatly reduce the complexity of a branch-and-bound solution, even in instances that are not particularly close to diagonal.  Similar gains are observed with respect to the mixed integer programming solver CPLEX.  Motivated by these results, the approximation properties of the diagonal relaxation are analyzed.  In particular, bounds on the approximation ratio are established in terms of the eigenvalues of the matrix defining the quadratic constraint, and also in the diagonally dominant and nearly coordinate-aligned cases.  
\end{abstract}

\section{Introduction}
\label{sec:intro}

This paper considers the problem of minimizing the cardinality of a vector $\mbx \in \mathbb{R}^{N}$ subject to a single convex quadratic constraint:
\begin{equation}\label{eqn:sparseProb}
\min_{\mbx} \quad C(\mbx) \qquad \text{s.t.} \qquad (\mbx - \mbc)^T \mbQ (\mbx - \mbc) \leq \gamma,
\end{equation}
where $C(\mbx)$ is the number of nonzero components of $\mbx$, $\mbQ$ is a positive definite matrix, and $\gamma$ is a positive scalar.  Geometrically, problem \eqref{eqn:sparseProb} corresponds to finding a point of minimal cardinality in an ellipsoid, denoted as $\mcE_{\mbQ}$, centered at the point $\mbc$.  The orientation and relative lengths of the ellipsoid axes are determined by the eigenvectors and eigenvalues of $\mbQ$ while $\gamma$ determines its absolute size.  %We will make reference to this ellipsoidal interpretation in Section \ref{sec:lowerBounds}.

The author's interest in \eqref{eqn:sparseProb} stems from the design of digital filters in signal processing (see \cite{wei2012a,wei2012} and the references therein).  In this context, $\mbx$ represents a vector of filter coefficients and cardinality minimization is motivated by the fact that the cost of implementing a filter is often dominated by arithmetic operations, especially in hardware.  The quadratic constraint represents a requirement on filter performance, for example a specified fidelity in approximating a desired frequency response or a bound on recovery error in the equalization of communication channels.

Problem \eqref{eqn:sparseProb} also has applications to subset selection for linear regression \cite{miller2002,das2008}, more specifically the overdetermined case in which $\mbQ$ is positive definite and less so the underdetermined case in which $\mbQ$ is rank-deficient and control of the cardinality is employed as a regularization.  A similar problem arises in optimal linear-quadratic control with cardinality-constrained input \cite{gao2011} (see also \cite{lin2012} for optimal control with sparse state-feedback gains).  A problem related to \eqref{eqn:sparseProb} has been studied extensively in cardinality-constrained financial portfolio optimization \cite{bertsimas2009,bienstock1996,bonami2009,cui2012,frangioni2006,gunluk2010,shaw2008,vielma2008}.  The portfolio optimization problem however has additional linear constraints, most notably non-negativity, upper bounds on nonzero variables, and sometimes lower bounds as well.  There is some computational evidence \cite{bertsimas2009} to suggest that the relative lack of constraints in \eqref{eqn:sparseProb} increases the difficulty of the problem, at least when approached using conventional integer optimization methods. %It is unclear whether the additional structure imposed by these linear constraints increases or alleviates the difficulty of the resulting discrete optimization.

Certain cases of \eqref{eqn:sparseProb} are known to be efficiently solvable, the simplest of which is the case of diagonal $\mbQ$.  Extensions to block-diagonal, tridiagonal, and well-conditioned $\mbQ$ are discussed in \cite{wei2012a}.  The authors of \cite{das2008} present polynomial algorithms for several additional cases, including an FPTAS for the general banded case and exact algorithms for the cases of a tree-structured covariance graph, a large independent set (``arrow''-structured $\mbQ$), and exponential decay in the entries of $\mbQ$ away from the diagonal.  The case in which nearly all of the eigenvalues of $\mbQ$ are identical and larger than the rest is treated in \cite{gao2012}.  

In the general case, \eqref{eqn:sparseProb} is a difficult combinatorial optimization problem.
% for which no polynomial-time algorithm is known.  
Several heuristics such as forward and backward greedy selection can be used, often with good results (see e.g.~\cite{wei2012a}, also \cite{zheng2012} for references on portfolio selection heuristics).  Although approximation guarantees do exist for forward selection in the near-diagonal case \cite{das2008} and for backward selection when a (difficult to evaluate) threshold test is met \cite{couvreur2000}, more general guarantees for heuristics are not available.  Thus if a certificate of optimality or a bound on the deviation from optimality is desired, branch-and-bound remains the method of choice and has therefore been considered by many researchers \cite{bienstock1996,shaw2008,bertsimas2009,bonami2009,vielma2008}.  In particular, \cite{bienstock1996} investigates a branch-and-cut algorithm employing disjunctive cuts and finds that such cuts are ineffective when $\mbQ$ is near full rank.  In \cite{bertsimas2009}, Lemke's pivoting method is used to provide warm starts in solving continuous relaxations.  Lagrangian relaxations have also been considered \cite{shaw2008}.  In \cite{vielma2008}, a lifted polyhedral relaxation is applied to mixed-integer second-order cone programs (which include the problems considered here) to take advantage of the more mature techniques for solving mixed-integer linear programs.

The complexity of branch-and-bound can be significantly reduced if specialized relaxations are available that can better approximate the original optimal cost while remaining efficiently solvable.  Such relaxations permit increased pruning of the branch-and-bound tree and can also suggest stronger reformulations of the original problem.  In the present context, a sequence of works \cite{frangioni2006,frangioni2007,gunluk2010,cui2012,zheng2012} have developed the \emph{perspective} relaxation, so-called because of its relationship to the perspective of a convex function.  The perspective relaxation can also be viewed as a particularly tractable instance of disjunctive convex optimization \cite{ceria1999}.  In \cite{frangioni2006,gunluk2010}, the relaxation is derived for general convex functions (not necessarily quadratic) using a convex hull approach; \cite{frangioni2006} emphasizes the identification of linear cuts whereas \cite{gunluk2010} proposes solving the nonlinear relaxation directly, aided by second-order cone representations.  In contrast, \cite{cui2012} focuses on portfolio optimization and derives the relaxation through Lagrangian decomposition.  The authors of \cite{frangioni2006,gunluk2010,cui2012} also show that the perspective relaxation is tighter than the standard continuous relaxation in certain contexts.  To apply the relaxation to portfolio optimization problems, a diagonal matrix must be separated from $\mbQ$; a semidefinite programming method for determining the best separation was reported very recently \cite{zheng2012} and is shown to outperform simpler methods in \cite{frangioni2006,frangioni2007}.  %largely focused on the cardinality-constrained portfolio optimization problem as opposed to \eqref{eqn:sparseProb}  
None of the above works however have analyzed the quality of approximation of the relaxation with respect to the original problem.

%At the same time, the bounds must be efficiently computable since they may potentially be evaluated for a large number of subproblems.

In this paper, we focus on the pure quadratically-constrained problem \eqref{eqn:sparseProb} and investigate two relaxations. % with the goal of obtaining bounds on the optimal cost.  Both relaxations are convex and can be efficiently solved.  
The first is the conventional continuous relaxation, obtained by formulating \eqref{eqn:sparseProb} as a mixed-integer optimization and relaxing binary-value constraints to unit interval constraints.  An absolute upper bound is given on the optimal cost of the continuous relaxation.  The bound suggests that the continuous relaxation is relatively weak for many instances of \eqref{eqn:sparseProb}, a hypothesis borne out by numerical experiments.  %We show analytically as well as numerically that linear relaxations yield relatively weak bounds for many instances of \eqref{eqn:sparseProb}.  
The second relaxation exploits the simplicity of the case of diagonal $\mbQ$, specifically by constructing the best diagonal approximation to \eqref{eqn:sparseProb}, referred to as a diagonal relaxation. %, when $\mbQ$ is non-diagonal.  
A computational comparison of the two relaxations shows that diagonal relaxations often yield significantly stronger bounds and can greatly decrease the complexity of a branch-and-bound solution to \eqref{eqn:sparseProb}, by orders of magnitude in difficult instances, and even when $\mbQ$ does not seem close to diagonal.  Similar efficiency gains are seen relative to the mixed-integer programming solver CPLEX \cite{cplex2012}.  Motivated by these results, this paper undertakes a theoretical analysis of diagonal relaxations, providing approximation guarantees for certain classes of instances and general insight into when diagonal relaxations are expected to be successful.  In particular, bounds on the approximation ratio are derived in terms of the eigenvalues of $\mbQ$ and in the cases of diagonally dominant $\mbQ$ and nearly coordinate-aligned $\mcE_{\mbQ}$.  We note that a relaxation similar to the diagonal relaxation was proposed independently in \cite{gao2011} with similarly positive computational experience.  A principal objective of the current paper is to support such findings with more detailed analysis.

We begin in Sect.~\ref{sec:prelim} by deriving some preliminary facts pertaining to problem \eqref{eqn:sparseProb}.  In Sect.~\ref{sec:linRelax}, continuous relaxations of \eqref{eqn:sparseProb} are discussed and analyzed, while the same is done for diagonal relaxations in Sect.~\ref{sec:diagRelax}.  In Sect.~\ref{sec:numEx}, the two relaxations are compared numerically in terms of their approximation ratios and effect on branch-and-bound complexity.  A comparison with CPLEX is also reported.  The paper concludes in Sect.~\ref{sec:concl}.

\subsection{Notation}
\label{subsec:notation}

Vectors and matrices are denoted using lowercase and uppercase boldface letters with $x_{n}$ representing the $n$th element of a vector $\mbx$ and $Q_{mn}$ the $(m,n)$ element of a matrix $\mbQ$.  
%Depending on context, a boldface $\0$ may denote the zero vector or the zero matrix of appropriate dimensions.
The letter $\mbe$ is reserved for a vector of unit entries. %; $\mbe_k$ represents the $k$th standard basis vector. 
%The $n$th element of a vector $\mbx$ is written $x_{n}$ and the $(m,n)$ element of a matrix $\mbQ$ is $Q_{mn}$.  
%The zero-norm notation $\norm[0]{\mbx}$ refers to the number of non-zero elements of $\mbx$.  
For sets of indices $Y$ and $Z$, $\mbx_Y$ represents the $\abs{Y}$-dimensional subvector of $\mbx$ corresponding to $Y$ and $\mbQ_{YZ}$ the $\abs{Y}\times\abs{Z}$ submatrix of $\mbQ$ with rows indexed by $Y$ and columns indexed by $Z$.  %Matrix transposition is denoted by a superscript $^T$.  
The notation $\mbQ \succeq \0$ ($\mbQ \succ \0$) indicates that $\mbQ$ is positive semidefinite (positive definite); $\mbQ \succeq \mbD$ is equivalent to $\mbQ - \mbD \succeq \0$.  The $n$th smallest eigenvalue of $\mbQ$ is written $\lambda_n(\mbQ)$ except as noted in Sect.~\ref{subsec:diagRelaxNAA}; we also use $\lambda_{\min}(\mbQ)$ and $\lambda_{\max}(\mbQ)$ for the smallest and largest eigenvalues.
%Given a vector $\mbc$, $\Diag(\mbc)$ represents a diagonal matrix with the entries of $\mbc$ along the diagonal; given a matrix $\mbH$, $\diag(\mbH)$ is the vector formed from the diagonal entries of $\mbH$.
%The notation $\abs{\mcY}$ refers to the number of elements in the set $\mcY$. 

%We denote the inner product $\tr(\mbP^T \mbQ)$ between matrices $\mbP$ and $\mbQ$ as $\mbP \bullet \mbQ$.

\section{Preliminaries}
\label{sec:prelim}

In this section, some facts related to problem \eqref{eqn:sparseProb} are derived for later use.  In Sect.~\ref{subsec:prelimFeasCond}, a condition is given for the feasibility of solutions of specified cardinality.  In Sect.~\ref{subsec:prelimElim}, it is shown that variables that are either constrained to a zero value or assumed to be nonzero can be eliminated to yield a lower-dimensional instance of \eqref{eqn:sparseProb}.

\subsection{Feasibility of solutions of specified cardinality}
\label{subsec:prelimFeasCond}

First we obtain a condition for the existence of feasible solutions to \eqref{eqn:sparseProb} with a specified number $K$ of zero-valued components.  Suppose that $x_{n}$ is constrained to a zero value for $n$ in a set $Z$ of size $K$.  With $Y$ denoting the complement of $Z$, the constraint in \eqref{eqn:sparseProb} becomes 
\begin{equation}\label{eqn:quadConsZ}
%\begin{multline}\label{eqn:quadCons1Z}
%\begin{bmatrix} (\mbb_\mcY - \mbc_\mcY)^T & -\mbc_\mcZ^T \end{bmatrix}
%\begin{bmatrix} \mbQ_{\mcY\mcY} & \mbQ_{\mcY\mcZ} \\ 
%\mbQ_{\mcZ\mcY} & \mbQ_{\mcZ\mcZ} \end{bmatrix}
%\begin{bmatrix} \mbb_\mcY - \mbc_\mcY \\ -\mbc_\mcZ \end{bmatrix}\\
%= 
(\mbx_Y - \mbc_Y)^T \mbQ_{YY} (\mbx_Y - \mbc_Y) - 2 \mbc_Z^T \mbQ_{ZY} (\mbx_Y - \mbc_Y) + \mbc_Z^T \mbQ_{ZZ} \mbc_Z \leq \gamma.
%\end{multline}
\end{equation}
%
%where $\mbQ_{\mcZ\mcY}$ denotes the submatrix of $\mbQ$ with rows indexed by $\mcZ$ and columns indexed by $\mcY$.  
Consider minimizing the left-hand side of \eqref{eqn:quadConsZ} with respect to $\mbx_{Y}$, with solution $\mbx_Y - \mbc_Y = (\mbQ_{YY})^{-1} \mbQ_{YZ} \mbc_Z$.  If \eqref{eqn:quadConsZ} is not satisfied when the left-hand side is minimized, then it cannot be satisfied for any value of $\mbx_{Y}$.  Hence a feasible solution to \eqref{eqn:sparseProb} exists subject to $x_{n} = 0$ for $n \in Z$ if and only if 
\begin{equation}\label{eqn:feasTestZ}
\mbc_Z^T ( \mbQ / \mbQ_{YY} ) \mbc_Z \leq \gamma,
\end{equation}
where $\mbQ / \mbQ_{YY} = \mbQ_{ZZ} - \mbQ_{ZY} (\mbQ_{YY})^{-1} \mbQ_{YZ} = \left( \bigl(\mbQ^{-1}\bigr)_{ZZ} \right)^{-1}$ is the Schur complement of $\mbQ_{YY}$.  Condition \eqref{eqn:feasTestZ} may be generalized to encompass all subsets of cardinality $K$ using a similar argument: If \eqref{eqn:feasTestZ} is not satisfied when the left-hand side is minimized over all subsets $Z$ of cardinality $K$, then there can be no solution to \eqref{eqn:sparseProb} with $K$ zero-valued components.  This yields the condition 
%, defined as \cite{hornjohnson1994}
%
%\begin{equation}\label{eqn:Schur}
%\[
%\mbQ / \mbQ_{\mcY\mcY} = \mbQ_{\mcZ\mcZ} - \mbQ_{\mcZ\mcY} (\mbQ_{\mcY\mcY})^{-1} \mbQ_{\mcY\mcZ} = \left( \bigl(\mbQ^{-1}\bigr)_{\mcZ\mcZ} \right)^{-1}.
%\]
%\end{equation}
%

%
\begin{equation}\label{eqn:feasTestK}
E_{0}(K) = \min_{\abs{Z}=K} \left\{ \mbc_Z^T( \mbQ / \mbQ_{YY} ) \mbc_Z \right\} \leq \gamma
\end{equation}
for the existence of a feasible solution of cardinality $N-K$.  In general, computing $E_{0}(K)$ in \eqref{eqn:feasTestK} involves an intractable combinatorial optimization.  However, when $\mbQ$ has special structure, $E_{0}(K)$ becomes much easier to evaluate and it is in these cases that condition \eqref{eqn:feasTestK} will be used.

In the special case of a single zero-value constraint, i.e., $Z = \{n\}$, condition \eqref{eqn:feasTestZ} reduces to 
\begin{equation}\label{eqn:feasTestSingleZero}
\frac{c_{n}^{2}}{\bigl(\mbQ^{-1}\bigr)_{nn}} \leq \gamma.
\end{equation}
If \eqref{eqn:feasTestSingleZero} is not satisfied, then it is not feasible for $x_{n}$ to take a value of zero.  It follows that an easily computed lower bound on the optimal cost in \eqref{eqn:sparseProb} is obtained by counting the number of indices $n$ for which \eqref{eqn:feasTestSingleZero} is not satisfied.  Furthermore, in Sect.~\ref{subsec:prelimElim}, it is shown that the variables $x_{n}$ corresponding to violations of \eqref{eqn:feasTestSingleZero} can be eliminated from the problem to reduce its dimension.

\subsection{Variable elimination}
\label{subsec:prelimElim}

We now consider restrictions of problem \eqref{eqn:sparseProb} in which certain variables are constrained to zero while others are assumed to be nonzero.  These two types of constraints arise in branch-and-bound as \eqref{eqn:sparseProb} is divided recursively into subproblems.  Variables that must be non-zero to maintain feasibility may also be identified through condition \eqref{eqn:feasTestSingleZero}.  

Let $Z$ denote as before the subset of variables constrained to zero, $U$ the subset of variables assumed to be nonzero, and $F$ the remainder.  We show that an arbitrary subproblem defined by subsets $(Z, U, F)$ can be reduced to the following problem:
\begin{equation}\label{eqn:subprobF}
\min_{\mbx_F} \quad \abs{U} + C(\mbx_F) \qquad 
\text{s.t.} \qquad \left( \mbx_F - \mbc_{\eff} \right)^T \mbQ_{\eff} \left( \mbx_F - \mbc_{\eff}
\right) \leq \gamma_{\eff},  
\end{equation}
with effective parameters given by  
\begin{subequations}\label{eqn:Qcgammaeff}
\begin{align}
\mbQ_{\eff} &= \mbQ_{FF} - \mbQ_{FU} \left(
\mbQ_{UU} \right)^{-1} \mbQ_{UF}, \label{eqn:Qeff}\\ 
\mbc_{\eff} &= \mbc_F + (\mbQ_{\eff})^{-1} \bigl( \mbQ_{FZ} - \mbQ_{FU} (\mbQ_{UU})^{-1} \mbQ_{UZ} \bigr) \mbc_Z, \label{eqn:ceff}\\
%\gamma_{\eff} &= \gamma - \mbc_\mcZ^T \left( \bigl(\mbQ^{-1}\bigr)_{\mcZ\mcZ} \right)^{-1} \mbc_\mcZ. \label{eqn:gammaeff} 
\gamma_{\eff} &= \gamma - \mbc_Z^T ( \mbQ / \mbQ_{YY} ) \mbc_Z. \label{eqn:gammaeff} \end{align}
\end{subequations}
Problem \eqref{eqn:subprobF} is an instance of \eqref{eqn:sparseProb} with $\abs{F}$ variables instead of $N$.  

The reduction can be carried out in the two steps $(\emptyset, \emptyset, \{1,\ldots,N\}) \longrightarrow (Z, \emptyset, Y = U \cup F) \longrightarrow (Z, U, F)$.  In the first step, the constraints $x_{n} = 0$ for $n \in Z$ reduce $C(\mbx)$ to $C(\mbx_{Y})$ and the quadratic constraint in \eqref{eqn:sparseProb} to \eqref{eqn:quadConsZ}.  By completing the square, \eqref{eqn:quadConsZ} can be rewritten as 
\begin{equation}\label{eqn:quadConsY}
\begin{bmatrix} \mbx_U - \mbc'_U \\ \mbx_F - \mbc'_F \end{bmatrix}^{T}
\begin{bmatrix} \mbQ_{UU} & \mbQ_{UF} \\ 
\mbQ_{FU} & \mbQ_{FF} \end{bmatrix}
\begin{bmatrix} \mbx_U - \mbc'_U \\ \mbx_F - \mbc'_F \end{bmatrix} 
\leq \gamma_{\eff},
%(\mbb_\mcY - \mbc'_\mcY)^T \mbQ_{\mcY\mcY} (\mbb_\mcY - \mbc'_\mcY) \leq \gamma_{\eff}
\end{equation}
where the subset $Y$ has been partitioned into $U$ and $F$, $\mbc'_U = \mbc_U + \left( (\mbQ_{YY})^{-1} \mbQ_{YZ} \mbc_Z \right)_U$, and $\mbc'_F = \mbc_F + \left( (\mbQ_{YY})^{-1} \mbQ_{YZ} \mbc_Z \right)_F$.  

In the second step $(Z, \emptyset, U\cup F) \longrightarrow (Z, U, F)$, the non-zero assumption on $\mbx_{U}$ allows $C(\mbx_{Y})$ to be rewritten as $\abs{U} + C(\mbx_F)$.  Since $\mbx_{U}$ no longer has any effect on the objective function, its value can be freely chosen, and in the interest of minimizing $C(\mbx_F)$, $\mbx_U$ should be chosen as a function of $\mbx_F$ to maximize the margin in constraint \eqref{eqn:quadConsY}, thereby making the set of feasible $\mbx_F$ as large as possible.  This is equivalent to minimizing the left-hand side of \eqref{eqn:quadConsY} with respect to $\mbx_U$ while holding $\mbx_F$ constant.  Similar to the minimization of \eqref{eqn:quadConsZ} with respect to $\mbx_{Y}$, we obtain $\mbx_U^\ast = \mbc'_U - \left( \mbQ_{UU} \right)^{-1} \mbQ_{UF} (\mbx_F - \mbc'_F)$ as the minimizer of \eqref{eqn:quadConsY}.  Substituting back into \eqref{eqn:quadConsY} results in the constraint in \eqref{eqn:subprobF}, except with $\mbc'_{F}$ in place of $\mbc_{\eff}$.  By expressing $(\mbQ_{YY})^{-1}$ in terms of the block decomposition of $\mbQ_{YY}$ in \eqref{eqn:quadConsY}, it can be shown that $\mbc'_{F}$ is equal to $\mbc_{\eff}$ as defined in \eqref{eqn:ceff}, thus completing the reduction.

In the sequel, we focus on the unrestricted root problem \eqref{eqn:sparseProb} with the understanding that the results apply to any subproblem by virtue of the reduction to \eqref{eqn:subprobF}.  In addition, the following assumption will be made:
\begin{assumption}\label{ass:feasTestSingleZero}
Condition \eqref{eqn:feasTestSingleZero} is satisfied for all $n = 1,\ldots,N$.
\end{assumption}
In other words, it is assumed that a feasible solution exists whenever a single variable is constrained to zero, since any variables for which this is not the case can be eliminated as shown in this section.  Thus the focus is solely on the ``difficult'' part of the problem, i.e., those variables whose status is ambiguous.

\section{Continuous relaxation}
\label{sec:linRelax}

In the remainder of the paper, we consider two relaxations of \eqref{eqn:sparseProb} for the purpose of obtaining lower bounds on its optimal cost in the context of branch-and-bound.  In Sect.~\ref{subsec:linRelaxDeriv}, \eqref{eqn:sparseProb} is reformulated as a mixed-integer optimization problem, yielding a continuous relaxation.  Best-case and worst-case instances are exhibited in Sect.~\ref{subsec:linRelaxEx} to show that continuous relaxations can provide arbitrarily tight or loose bounds on the optimal cost of \eqref{eqn:sparseProb}.  An absolute upper bound on the optimal cost of the relaxation is then derived in Sect.~\ref{subsec:linRelaxBounds}, suggesting that continuous relaxations are unlikely to yield good approximations to \eqref{eqn:sparseProb} in most instances.

\subsection{Derivation}
\label{subsec:linRelaxDeriv}

Problem \eqref{eqn:sparseProb} is first reformulated as a mixed integer optimization problem by associating with each continuous variable $x_{n}$ a binary-valued indicator variable $i_n$ with the property that $i_n = 0$ if $x_n = 0$ and $i_n = 1$ otherwise.  Problem \eqref{eqn:sparseProb} can be restated in terms of indicator variables as follows:
\begin{equation}\label{eqn:MIP}
%\begin{split}
\min_{\mbx,\mbi} \quad \sum_{n=1}^{N} i_n \qquad 
\text{s.t.} \qquad (\mbx - \mbc)^T \mbQ (\mbx - \mbc) \leq \gamma, \quad
\abs{x_n} \leq B_n i_n, \quad
i_n \in \{0, 1\} \quad \forall \; n.
%\end{split}
\end{equation}
The constraint $\abs{x_n} \leq B_n i_n$ is the usual forcing constraint linking $i_{n}$ with $x_{n}$ in the desired manner, where the positive constants $B_n$ are chosen large enough to keep the set of feasible $\mbx$ unchanged from that in \eqref{eqn:sparseProb}.  It will be seen shortly that $B_{n}$ should be set to the smallest possible value subject to this requirement, i.e.,  
%
%\begin{equation}\label{eqn:Bn}
\[
B_n = \max \left\{ \abs{x_n} : (\mbx - \mbc)^T \mbQ (\mbx - \mbc) \leq \gamma \right\} = \max \left\{ B_{n}^{+}, B_{n}^{-} \right\},
\]
%\end{equation}
%
where
%The inner maximizations in \eqref{eqn:Bn} can be solved in closed form as shown in Appendix \ref{app:Bn+-}, yielding
%
\begin{equation}\label{eqn:Bn+-}
B_{n}^{\pm} = \max \left\{ \pm x_n : (\mbx - \mbc)^T \mbQ (\mbx - \mbc) \leq \gamma \right\} = \sqrt{\gamma \bigl(\mbQ^{-1}\bigr)_{nn}} \pm c_n.
\end{equation}
The closed-form expressions for $B_{n}^{+}$ and $B_{n}^{-}$ can be derived straightforwardly from the corresponding KKT conditions \cite{bertsekas1999,wei2011}.  %Hence \eqref{eqn:Bn} simplifies to $B_n^{\ast} = \sqrt{\gamma \bigl(\mbQ^{-1}\bigr)_{nn}} + \abs{c_n}$.

A continuous relaxation of \eqref{eqn:MIP} results from relaxing the binary-value constraints on $i_n$ to interval constraints $0 \leq i_n \leq 1$.  By minimizing the objective with respect to $\mbi$ and substituting back into \eqref{eqn:MIP}, we obtain the following minimization with respect to $\mbx$:
\begin{equation}\label{eqn:linRelaxUnseparated}
\min_{\mbx} \quad \sum_{n=1}^{N} \frac{\abs{x_n}}{B_n} \qquad \text{s.t.} \qquad (\mbx - \mbc)^T \mbQ (\mbx - \mbc) \leq \gamma.
\end{equation}
The continuous relaxation \eqref{eqn:linRelaxUnseparated} is a quadratically-constrained weighted $1$-norm minimization and is therefore a convex problem.  The optimal cost in \eqref{eqn:linRelaxUnseparated} is clearly a lower bound on the optimal cost in \eqref{eqn:MIP} since the feasible set has been enlarged; more precisely, since the latter must be an integer, the ceiling of the former is also a lower bound.  It is also seen that the lower bound is maximized when the constants $B_n$ are as small as possible.  

A stronger lower bound on \eqref{eqn:MIP} can be obtained by first separating each variable $x_{n}$ into its positive and negative parts $x_{n}^{+}$ and $x_{n}^{-}$ as follows:
%The value of $B_n$ can be decreased even further if it can be made to depend on the sign of $b_n$.  For example, if $b_n$ is known to be positive, $B_n$ only needs to be greater than or equal to the quantity in \eqref{eqn:Bn+} without regard to \eqref{eqn:Bn-}, while the reverse is true if $b_n$ is negative.  Unless $c_n = 0$, one of the quantities in \eqref{eqn:Bn+} and \eqref{eqn:Bn-} is lower than the value in \eqref{eqn:Bn2}.  The key to allowing $B_n$ to vary in a sign-dependent way is to separate $b_n$ into its positive and negative parts.  We express each $b_n$ as
%
\begin{equation}\label{eqn:bn+-}
x_n = x_n^+ - x_n^-, \quad x_n^+, \; x_n^- \geq 0.
\end{equation}
%
%Under the condition that at least one of $b_n^+$, $b_n^-$ is always zero, the representation in \eqref{eqn:bn+-} is unique, $b_n = b_n^+$ for $b_n > 0$, and $b_n = -b_n^-$ for $b_n < 0$.  
By assigning to each pair $x_n^+$, $x_n^-$ corresponding indicator variables $i_n^+$, $i_n^-$ and constants $B_n^+$, $B_n^-$, a mixed integer optimization problem equivalent to \eqref{eqn:MIP} may be formulated, where the values of $B_n^+$ and $B_n^-$ are given by \eqref{eqn:Bn+-}.  The continuous relaxation of this alternative mixed integer formulation corresponds to the following quadratically-constrained linear program:
\begin{equation}\label{eqn:linRelaxPrimal}
%\begin{split}
\min_{\mbx^+,\mbx^-} \quad \sum_{n=1}^{N} \left(\frac{x_n^+}{B_n^+} 
+ \frac{x_n^-}{B_n^-} \right) \qquad 
\text{s.t.} \qquad (\mbx^+ - \mbx^- - \mbc)^T \mbQ (\mbx^+ - \mbx^- -
\mbc) \leq \gamma, \quad
\mbx^{\pm} \geq \0.
%\end{split}
\end{equation}
%
%Problem \eqref{eqn:linRelaxPrimal} is a quadratically constrained linear program and is also efficiently solvable.  The smallest values for $B_{n}^{+}$ and $B_{n}^{-}$ that ensure that \eqref{eqn:linRelaxPrimal} is a valid relaxation are given by $B_{n}^{+\ast}$ and $B_{n}^{-\ast}$ in \eqref{eqn:Bn+-}.  
Using \eqref{eqn:bn+-} to replace the absolute value functions in \eqref{eqn:linRelaxUnseparated} with linear functions as done in linear programming \cite{bt1997}, it can be seen that \eqref{eqn:linRelaxUnseparated} is a special case of \eqref{eqn:linRelaxPrimal} with $B_n^+$ and $B_n^-$ replaced by $B_n$.  Since $B_n = \max\{ B_{n}^{+}, B_{n}^{-} \}$, the optimal cost in \eqref{eqn:linRelaxPrimal} is at least as large as that in \eqref{eqn:linRelaxUnseparated}, and therefore \eqref{eqn:linRelaxPrimal} is at least as strong a relaxation as \eqref{eqn:linRelaxUnseparated}.  The term continuous relaxation will refer henceforth to \eqref{eqn:linRelaxPrimal} with $B_{n}^{\pm}$ given by \eqref{eqn:Bn+-}.

%An alternative interpretation of the linear relaxation in \eqref{eqn:linRelaxUnseparated} is as a weighted $\ell^1$ relaxation of the $\ell^0$ minimization in \eqref{eqn:sparseProb}.  In \eqref{eqn:linRelaxPrimal}, the weights are also allowed to depend on the signs of the coefficients.  The values of $B_n^+$ and $B_n^-$ in \eqref{eqn:Bn+-2} correspond to the choice of weights that gives the tightest relaxation in this class.  For this reason, we will use the term linear relaxation to refer henceforth to \eqref{eqn:linRelaxPrimal} with $B_n^+$ and $B_n^-$ given by \eqref{eqn:Bn+-2}.

Fig.~\ref{fig:weightedl1} shows a graphical interpretation of the continuous relaxation \eqref{eqn:linRelaxPrimal}.  The asymmetric diamond represents a level contour of the cost function, which can be regarded as a weighted $1$-norm with different weights for positive and negative component values.  As seen from \eqref{eqn:Bn+-}, the weights $B_n^{\pm}$ correspond to the maximum extent of the ellipsoid $\mcE_{\mbQ}$ along the positive and negative coordinate directions and can be found graphically as indicated in Fig.~\ref{fig:weightedl1}.  The solution to the weighted $1$-norm minimization can be visualized by inflating the diamond until it just touches the ellipsoid.  Note that Assumption \ref{ass:feasTestSingleZero} implies that $\mcE_{\mbQ}$ must intersect all of the coordinate planes.  In Sect.~\ref{subsec:linRelaxEx}, we will draw upon the geometric intuition in Fig.~\ref{fig:weightedl1} to construct best-case and worst-case instances for continuous relaxation. %The optimal solution is given by the point of tangency and the optimal value by the tangent contour.  

\begin{figure}[ht]
\centering
\psfrag{b1}[][]{$b_1$}
\psfrag{b2}[][]{$b_2$}
\psfrag{B1+}[][]{$B_1^+$}
\psfrag{B1-}[][]{$B_1^-$}
\psfrag{B2+}[][]{$B_2^+$}
\psfrag{B2-}[][]{$B_2^-$}
\psfrag{EQ}[][]{$\mcE_{\mbQ}$}
\includegraphics[width=0.45\columnwidth]{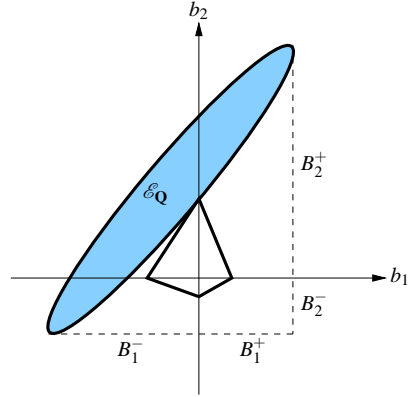}
\caption{Interpretation of the continuous relaxation as a weighted $1$-norm minimization and a graphical representation of its solution.} 
\label{fig:weightedl1}
\end{figure}

\begin{comment}
In Section \ref{subsec:linRelaxEx}, we will also require the dual of the linear relaxation \eqref{eqn:linRelaxPrimal}, given by 
%
\begin{equation}\label{eqn:linRelaxDual}
\begin{split}
\max_{\mbmu} \quad &\mbc^T \mbmu - \sqrt{\gamma \mbmu^T \mbQ^{-1}
 \mbmu}\\  
\text{s.t.} \quad &-\mbg^- \leq \mbmu \leq \mbg^+,
\end{split}
\end{equation}
%
with $g_n^+ = 1/B_n^+$ and $g_n^- = 1/B_n^-$ for all $n$.  The derivation of the dual problem can be found in Appendix \ref{app:linRelaxDual}.  Since the primal problem \eqref{eqn:linRelaxPrimal} is convex and has a strictly feasible solution $\mbb^+ - \mbb^- = \mbc$, and the dual has a strictly feasible solution $\mbmu = \0$, by Slater's condition the optimal values of the primal and dual are equal \cite{bertsekas1999}.  The dual is a nonlinear optimization problem with upper and lower bound constraints on each of the variables and is generally easier to solve than the primal because of the nature of the constraints. 
%A customized algorithm that solves the linear relaxation efficiently will be described in a future paper.
\end{comment}

\subsection{Best-case and worst-case instances}
\label{subsec:linRelaxEx}

%In general, given a relaxation of an optimization problem, it is of interest to analyze the conditions under which the relaxation is either a good or a poor approximation to the original problem.  
In this subsection, instances of problem \eqref{eqn:sparseProb} are exhibited to show that the continuous relaxation can be a tight approximation to \eqref{eqn:sparseProb} as well as an arbitrarily poor one.  The quality of approximation is characterized by the approximation ratio, defined as the ratio of the optimal cost of the relaxation to the optimal cost of the original problem.  

In the instances to be constructed, we take $\mbc = \mbe$ and $\gamma = 1$, which can be regarded as a normalization.  The matrix $\mbQ$ is restricted to be of the form 
\begin{equation}\label{eqn:Qex}
\mbQ = \lambda_{2} \mbI - (\lambda_{2} - \lambda_{1}) \mbv \mbv^{T},
\end{equation}
where $\lambda_{2} > \lambda_{1}$ and $\mbv$ is vector with unit 2-norm and components equal to $\pm 1/\sqrt{N}$.  It follows from \eqref{eqn:Qex} that $\mbv$ is an eigenvector of $\mbQ$ with eigenvalue $\lambda_{1}$ and the remaining $N-1$ eigenvectors are orthogonal to $\mbv$ with eigenvalue $\lambda_{2}$.  Geometrically, the ellipsoid $\mcE_{\mbQ}$ corresponding to \eqref{eqn:Qex} has a single long principal axis in the direction $\mbv$ and shorter and equal principal axes in the other directions.  We note for later use that the inverse of $\mbQ$ and the Schur complement $\mbQ / \mbQ_{YY}$ can be computed explicitly as 
\begin{align}
\mbQ^{-1} &= \frac{1}{\lambda_{2}} \mbI + \frac{\lambda_{2} - \lambda_{1}}{\lambda_{1} \lambda_{2}} \mbv \mbv^{T},\label{eqn:QinvEx}\\
\mbQ / \mbQ_{YY} &= \lambda_{2} \mbI - \frac{K \lambda_{2} (\lambda_{2} - \lambda_{1})}{K \lambda_{2} + (N-K) \lambda_{1}} \hat{\mbv}_{Z} \hat{\mbv}_{Z}^{T},\label{eqn:QschurEx}
\end{align}
where $K = \abs{Z}$ and $\hat{\mbv}_{Z}$ is the unit 2-norm vector obtained by rescaling the subvector $\mbv_{Z}$.

To construct best-case instances for which the continuous relaxation is a tight approximation to \eqref{eqn:sparseProb}, our aim is to make the optimal cost of the relaxation as large as possible.  Based on Fig.~\ref{fig:weightedl1} and the above structure for $\mbQ$, this can be done by choosing the major axis of $\mcE_{\mbQ}$ to be parallel to a level surface of the $1$-norm and keeping the lengths of the minor axes to a minimum, thus allowing the $\ell_{1}$ ball to grow relatively unimpeded.  Algebraically, we set $\lambda_{1} = 1/N$, $\lambda_{2} = N$, $\lceil N/2 \rceil$ of the components of $\mbv$ equal to $+1/\sqrt{N}$, and the remaining components of $\mbv$ equal to $-1/\sqrt{N}$.  

First it is shown that the point $\mbc - \sqrt{N} \mbv$ is optimal for \eqref{eqn:sparseProb} with a corresponding cost of $\lfloor N/2 \rfloor$.  Feasibility follows from substitution into the constraint in \eqref{eqn:sparseProb}.  To prove optimality, we verify that an additional zero-valued component is not feasible, i.e., condition \eqref{eqn:feasTestK} is violated for $K = N - \lfloor N/2 \rfloor + 1 = \lceil N/2 \rceil + 1$.  Substituting \eqref{eqn:QschurEx} and $\mbc = \mbe$ into \eqref{eqn:feasTestK} and rearranging, we obtain 
\begin{equation}\label{eqn:feasTestLinRelaxBestCase}
E_{0}(K) = (\lceil N/2 \rceil + 1)\lambda_{2} \left( 1 - \frac{(\lceil N/2 \rceil + 1) (\lambda_{2} - \lambda_{1})}{(\lceil N/2 \rceil + 1) \lambda_{2} + (\lfloor N/2 \rfloor - 1) \lambda_{1}} \max_{\abs{Z} = \lceil N/2 \rceil + 1} \frac{(\mbe^{T} \hat{\mbv}_{Z})^{2}}{\lceil N/2 \rceil + 1} \right). 
\end{equation}
The maximum in \eqref{eqn:feasTestLinRelaxBestCase} is achieved by choosing $Z$ to include all $\lceil N/2 \rceil$ positive components of $\mbv$ and only one negative component, resulting in a maximum value of $(\lceil N/2 \rceil - 1)^{2} / (\lceil N/2 \rceil + 1)^{2}$.  The quantity $E_{0}(K)$ can then be bounded from below by removing the fraction in front of the maximization.  This yields 
\[
(\lceil N/2 \rceil + 1)\lambda_{2} \left( 1 - \frac{(\lceil N/2 \rceil - 1)^{2}}{(\lceil N/2 \rceil + 1)^{2}} \right),
\]
which can be seen to be strictly greater than $\gamma = 1$ as required.

We now prove that the lower bound provided by the continuous relaxation is equal to the optimal cost of $\lfloor N/2 \rfloor$ for the unrelaxed problem.  Toward this end, we make use of the Lagrangian dual of the continuous relaxation, given by 
\begin{equation}\label{eqn:linRelaxDual}
\max_{\mbmu} \quad \mbc^{T} \mbmu - \sqrt{\gamma \mbmu^{T} \mbQ^{-1} \mbmu} 
\qquad \text{s.t.} \qquad -\mbg^{-} \leq \mbmu \leq \mbg^{+},  
\end{equation}
where $g_{n}^{\pm} = 1/B_{n}^{\pm}$ for $n = 1,\ldots,N$.  A derivation of the dual problem can be found in \cite{wei2011}.  It is shown that the optimal cost of the dual is strictly bounded from below by $\lfloor N/2 \rfloor - 1$, implying through duality that the optimal cost of the primal \eqref{eqn:linRelaxPrimal} is between $\lfloor N/2 \rfloor - 1$ and $\lfloor N/2 \rfloor$ and is equal to $\lfloor N/2 \rfloor$ after rounding up to the next integer.  From \eqref{eqn:Bn+-} and \eqref{eqn:QinvEx}, we find that $B_{n}^{+} = 1 + \sqrt{1 + (N-1)/N^{2}} = B^{+}$ for all $n$.  Substituting the dual feasible solution $\mbmu = \mbg^{+} = (1 / B^{+}) \mbe$ into the dual objective function and simplifying, we obtain 
\begin{equation}\label{eqn:linRelaxBestCaseLB}
\begin{cases}
\frac{1}{B^{+}} (N-1), & N \text{ even},\\
\frac{1}{B^{+}} \left(N - \sqrt{2 - \frac{1}{N^{2}}} \right), & N \text{ odd},
\end{cases}
\end{equation}
as a lower bound on the dual optimal cost.  Straightforward algebraic manipulations show that the quantities in \eqref{eqn:linRelaxBestCaseLB} are strictly greater than $\lfloor N/2 \rfloor - 1$ in the two cases of $N$ even and $N$ odd.  This completes the demonstration of the potential tightness of the continuous relaxation lower bound.

Next we construct instances for which the lower bound resulting from the continuous relaxation is as loose as possible.  The worst-case scenario corresponds to the optimal cost in \eqref{eqn:sparseProb} being equal to $N-1$ and the optimal cost of the relaxation being less than $1$.  The former cannot equal $N$ given Assumption \ref{ass:feasTestSingleZero} while the latter cannot equal zero exactly since that would require $\mbx = \0$ to be a feasible solution, in which case the optimal cost in \eqref{eqn:sparseProb} is also zero.  Referring again to Fig.~\ref{fig:weightedl1} and the form of $\mbQ$ in \eqref{eqn:Qex}, the optimal cost of the continuous relaxation can be minimized by orienting the major axis of the ellipsoid $\mcE_{\mbQ}$ so that it points toward the origin and obstructs the growth of the $\ell_{1}$ ball.  Algebraically, we set $\mbv = (1/\sqrt{N}) \mbe$, $\lambda_{1} = 1/(N-1)$, and $\lambda_{2} = (N-1)/2$.  We verify that the unrelaxed optimal cost is equal to $N-1$.  From \eqref{eqn:QinvEx}, we have $(\mbQ^{-1})_{nn} = (N+1)/N$, which ensures that \eqref{eqn:feasTestSingleZero} is satisfied for all $n$.  Using \eqref{eqn:QschurEx}, $E_{0}(K)$ in \eqref{eqn:feasTestK} for $K = 2$ evaluates to $N(N-1) / (N(N-1) - 1)$.  Since this quantity is greater than $\gamma = 1$, condition \eqref{eqn:feasTestK} is violated for $K = 2$ and the optimal cost in \eqref{eqn:sparseProb} must be equal to $N-1$.

To show that the optimal cost of the continuous relaxation is less than $1$, we consider the feasible and strictly positive solution $\mbx^{+} = \mbc - (1/\sqrt{\lambda_{1}}) \mbv = (1 - \sqrt{(N-1)/N}) \mbe$, $\mbx^{-} = \0$.  From \eqref{eqn:linRelaxPrimal}, the corresponding cost is 
\begin{equation}\label{eqn:linRelaxWorstCaseUB}
\frac{N - \sqrt{N(N-1)}}{B^{+}},
\end{equation}
where $B^{+} = 1 + \sqrt{(N+1)/N}$ is the common value for the constants $B_{n}^{+}$ given by \eqref{eqn:Bn+-}.  Since $B^{+} > 2$ while the numerator in \eqref{eqn:linRelaxWorstCaseUB} can be seen to be less than $1$, we conclude that the optimal cost in \eqref{eqn:linRelaxPrimal} is less than $1$ as claimed.  The approximation ratio in these instances is thus equal to $1/(N-1)$, which approaches zero as $N$ increases.

\subsection{An absolute upper bound}
\label{subsec:linRelaxBounds}

\begin{comment}
In this section, we derive bounds on the optimal value of the linear
relaxation in \eqref{eqn:linRelaxPrimal}.
First we make use of the dual problem \eqref{eqn:linRelaxDual} to
obtain a lower bound on the optimal cost.
%thus giving an approximate analytical guarantee on the quality of the
%linear relaxation. 
This can be done by evaluating the cost of any feasible solution to
the dual. 
For example, setting $\mbmu$ to be proportional to $\mbf = \mbQ \mbc$
results in a value of 
%
\begin{equation}\label{eqn:linRelaxLB}
\alpha \mbc^T \mbQ \mbc - \alpha\sqrt{\gamma \mbc^T \mbQ \mbc} 
= \alpha \sqrt{\mbc^T \mbQ \mbc} \left( \sqrt{\mbc^T \mbQ \mbc} -
\sqrt{\gamma} \right),
\end{equation}
%
where $\alpha$ is the constant of proportionality.
As long as $\mbb^+ - \mbb^- = \0$ is not a feasible solution to
\eqref{eqn:MIP+-}, $\mbc^T \mbQ \mbc > \gamma$ and hence the quantity
in parentheses in \eqref{eqn:linRelaxLB} is positive. 
The constant $\alpha$ may be set to the largest value that keeps
$\mbmu = \alpha\mbf$ feasible for problem \eqref{eqn:linRelaxDual}.
Specifically, the constraints $\alpha f_n \leq g_n^+$ for $f_n
\geq 0$ and $\alpha \abs{f_n} \leq g_n^-$ for $f_n < 0$ imply that
$\alpha$ should be chosen as 
%
\[
\alpha = \min_{n:f_n \neq 0} \left[ \abs{f_n} \left(
\sqrt{\gamma \bigl(\mbQ^{-1}\bigr)_{nn}} + \sgn(f_n) c_n \right)
\right]^{-1},
\]
%
where we have used \eqref{eqn:Bn+-2}.

Conversely, any feasible solution to the primal problem
\eqref{eqn:linRelaxPrimal} provides an upper bound on the optimal cost
of the linear relaxation.
\end{comment}

The constructions in Sect.~\ref{subsec:linRelaxEx} imply that the approximation ratio for the continuous relaxation can range anywhere between $0$ and $1$, and thus it is not possible to place a non-trivial bound on the ratio that holds for all instances of \eqref{eqn:sparseProb}.  It is possible however to obtain an absolute upper bound on the optimal cost of the continuous relaxation in terms of the problem dimension $N$.  
\begin{proposition}\label{prop:linRelaxUB}
Under Assumption \ref{ass:feasTestSingleZero}, the optimal cost of the continuous relaxation \eqref{eqn:linRelaxPrimal} is bounded from above by $\theta N/2$, where $\theta = 1 - \sqrt{\gamma / \mbc^{T} \mbQ \mbc}$.
\end{proposition}
\begin{proof}
Consider the solution $\mbb^+ - \mbb^- = \theta\mbc$, i.e., $b_n^+ = \theta c_n$, $b_n^- = 0$ for $c_n \geq 0$ and $b_n^+ = 0$, $b_n^- = \theta\abs{c_n}$ for $c_n < 0$.  It can be verified that this is a feasible solution for the continuous relaxation (the solution lies on the boundary of the ellipsoid $\mcE_{\mbQ}$), and hence the optimal cost of the relaxation is bounded from above by 
\begin{equation}\label{eqn:linRelaxUB1}
\theta \sum_{n:c_n>0} \frac{c_n}{B_n^{+}} + \theta \sum_{n:c_n<0} \frac{\abs{c_n}}{B_n^{-}} 
= \theta \sum_{n=1}^{N} \frac{\abs{c_n}}{\sqrt{\gamma \bigl(\mbQ^{-1}\bigr)_{nn}} + \abs{c_n}},
\end{equation}
using \eqref{eqn:Bn+-}.  Assumption \ref{ass:feasTestSingleZero} then implies that each of the fractions on the right-hand side of \eqref{eqn:linRelaxUB1} is no greater than $1/2$, completing the proof. 
\qed\end{proof}
%($\abs{\mcF}/2$ in the case of problem \eqref{eqn:MIsubprobF1})

Proposition \ref{prop:linRelaxUB} indicates that the continuous relaxation cannot be a tight approximation if the optimal cost in \eqref{eqn:sparseProb} is greater than $\lceil \theta N/2 \rceil$. This suggests that it is unlikely for the continuous relaxation to yield a strong bound on \eqref{eqn:sparseProb} in most instances, since if it did, this would imply that the optimal cost in \eqref{eqn:sparseProb} is not much greater than $\theta N/2$ in most cases, a fact considered unlikely. 
%The upper bound in \eqref{eqn:linRelaxUB1} shows that the linear
%relaxation in \eqref{eqn:linRelaxPrimal} is unlikely to be a good
%approximation to \eqref{eqn:MIP+-} when $N$ is large and the true
%optimal cost is greater than $N/2$.
The situation is exacerbated if the factor $\theta$ is small.  This negative result motivates the consideration of an alternative relaxation as described in Section \ref{sec:diagRelax}.

We note in closing that Lemar\'{e}chal and Oustry \cite{lemarechaloustry1999} have shown that a common semidefinite relaxation technique is equivalent to continuous relaxation when applied to cardinality minimization problems such as \eqref{eqn:sparseProb}.  As a consequence, the properties of the continuous relaxation \eqref{eqn:linRelaxPrimal} noted in this section also apply to this type of semidefinite relaxation.

\section{Diagonal relaxation}
\label{sec:diagRelax}

As an alternative to continuous relaxations, in this section we discuss relaxations of problem \eqref{eqn:sparseProb} in which the matrix $\mbQ$ is replaced by a diagonal matrix, an approach referred to as diagonal relaxation.  As will be seen in Sect.~\ref{subsec:diagRelaxDeriv}, problem \eqref{eqn:sparseProb} is easily solved in the diagonal case, thus making it attractive as a relaxation of the problem when $\mbQ$ is non-diagonal.  %In Section \ref{subsec:diagRelaxDeriv}, we show how diagonal relaxations of \eqref{eqn:sparseProb} may be obtained, in particular the tightest possible diagonal relaxation.
It is shown in Sect.~\ref{subsec:diagRelaxWorstCase} that diagonal relaxations can yield exact as well as arbitrarily poor approximations to \eqref{eqn:sparseProb}, as was the case for the continuous relaxation in Sect.~\ref{sec:linRelax}.  However, numerical evidence in Sect.~\ref{sec:numEx} and elsewhere \cite{wei2012} indicates that the lower bounds provided by diagonal relaxations are often significantly stronger than those from continuous relaxations.  This computational experience motivates a better theoretical understanding of situations to which diagonal relaxations are particularly well-suited.  Within this context, approximation guarantees are derived in Sect.~\ref{subsec:diagRelaxEig}--\ref{subsec:diagRelaxNAA} for the three specific cases of well-conditioned $\mbQ$ matrices, diagonally dominant $\mbQ$, and nearly coordinate-aligned ellipsoids $\mcE_{\mbQ}$.

\subsection{Derivation}
\label{subsec:diagRelaxDeriv}

To obtain a diagonal relaxation of problem \eqref{eqn:sparseProb}, the matrix $\mbQ$ is replaced with a positive definite diagonal matrix $\mbD$ to yield a similar constraint:
\begin{equation}\label{eqn:quadConsDiag}
(\mbx - \mbc)^T \mbD (\mbx - \mbc) = \sum_{n=1}^{N} D_{nn} (x_n - c_n)^2 \leq \gamma.
\end{equation}
Geometrically, \eqref{eqn:quadConsDiag} specifies an ellipsoid, denoted as $\mcE_\mbD$, with axes that are aligned with the coordinate axes.  Since the relaxation is intended to provide a lower bound for the original problem, we require that the coordinate-aligned ellipsoid $\mcE_\mbD$ enclose the original ellipsoid $\mcE_\mbQ$ so that minimizing over $\mcE_\mbD$ yields a lower bound on the minimum over $\mcE_\mbQ$.  For simplicity, the two ellipsoids are taken to be concentric, in which case the nesting of the ellipsoids is equivalent to the condition $\mbD \preceq \mbQ$.  Sufficiency follows from the inequality
\begin{equation}\label{eqn:QDpsd}
(\mbx - \mbc)^T \mbD (\mbx - \mbc) \leq (\mbx - \mbc)^T \mbQ (\mbx - \mbc) \quad \forall \; \mbx,
\end{equation}
so if $\mbx \in \mcE_\mbQ$, then both sides of \eqref{eqn:QDpsd} are bounded by $\gamma$ and $\mbx \in \mcE_\mbD$.  Conversely, if $\mbD \not\preceq \mbQ$, then there exists a vector $\mbx$ that violates \eqref{eqn:QDpsd}, and by scaling $\mbx - \mbc$ so that the right-hand side of \eqref{eqn:QDpsd} is equal to $\gamma$, we have $\mbx \in \mcE_{\mbQ}$ but $\mbx \notin \mcE_{\mbD}$ since $\mbx$ does not satisfy \eqref{eqn:quadConsDiag}.

Problem \eqref{eqn:sparseProb} is greatly simplified in the diagonal case.  Replacing $\mbQ$ by $\mbD$, condition \eqref{eqn:feasTestK} simplifies to 
\[
\min_{\abs{Z} = K} \sum_{n\in Z} D_{nn} c_{n}^{2} \leq \gamma
\]
since $\mbD / \mbD_{YY} = \mbD_{ZZ}$.  The minimum is attained by choosing $Z$ to correspond to the $K$ smallest elements of the sequence $D_{nn} c_{n}^{2}, n = 1,\ldots,N$.  It follows that \eqref{eqn:quadConsDiag} admits a solution with $K$ zero-valued components if and only if 
\begin{equation}\label{eqn:feasTestKDiag}
S_K\bigl(\{ D_{nn} c_n^2 \}\bigr) \leq \gamma,
\end{equation}
where $S_{K}$ denotes the sum of the $K$ smallest elements of a sequence.  The minimum cardinality corresponds to the largest value of $K$ such that \eqref{eqn:feasTestKDiag} holds.

\begin{figure}[ht]
\centering
\psfrag{EQ}[][]{$\mcE_\mbQ$}
\psfrag{ED1}[][]{$\mcE_{\mbD_1}$}
\psfrag{ED2}[][]{$\mcE_{\mbD_2}$}
\includegraphics[width=0.4\columnwidth]{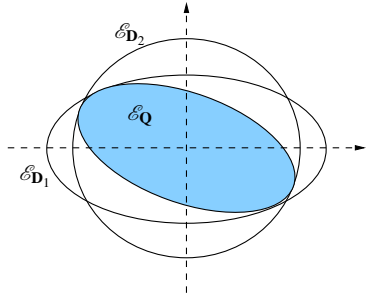}
\caption{Two different diagonal relaxations.}
\label{fig:diagRelax}
\end{figure}

For every $\mbD$ satisfying $\0 \preceq \mbD \preceq \mbQ$, minimizing $C(\mbx)$ subject to \eqref{eqn:quadConsDiag} results in a lower bound on the optimal cost in \eqref{eqn:sparseProb}.  Thus the set of diagonal relaxations is parameterized by $\mbD$ as illustrated in Fig.~\ref{fig:diagRelax}.  We are naturally interested in obtaining a diagonal relaxation that is as tight as possible, i.e., a matrix $\mbD_d$ such that the minimum cardinality associated with $\mbD_d$ is maximal among all valid choices of $\mbD$.  Such a relaxation can be determined based on condition \eqref{eqn:feasTestKDiag}, specifically by solving the following optimization problem: 
\begin{equation}\label{eqn:maxSumSmallest}
%\begin{split}
E_{d}(K) = \max_{\mbD} \quad S_K \bigl( \{D_{nn} c_n^2\} \bigr) \qquad 
\text{s.t.} \qquad \0 \preceq \mbD \preceq \mbQ, \quad \mbD \text{ diagonal},
%\end{split}
\end{equation}
for selected values of $K$.  If $E_d(K)$ in \eqref{eqn:maxSumSmallest} is less than or equal to $\gamma$, then \eqref{eqn:feasTestKDiag} holds for every $\mbD$ satisfying the constraints in \eqref{eqn:maxSumSmallest}, and consequently a feasible solution $\mbx$ with $K$ zero-valued components exists for every such $\mbD$.  We conclude that the optimal cost of any diagonal relaxation is at most $N - K$.  On the other hand, if $E_d(K) > \gamma$, then according to \eqref{eqn:feasTestKDiag} there exists a $\mbD$ for which a vector $\mbx$ with $K$ zero-valued components is not feasible, and for this $\mbD$ the optimal cost of the corresponding diagonal relaxation is at least $N - K + 1$.  By selecting values of $K$ to perform a bisection search over $1,\ldots,N$ and solving \eqref{eqn:maxSumSmallest} each time, we eventually arrive at the highest possible optimal cost under any diagonal relaxation, i.e., the tightest lower bound on \eqref{eqn:sparseProb} achievable with a diagonal relaxation.  Henceforth the term diagonal relaxation will be understood to refer to the tightest such relaxation.  

The above procedure determines both the tightest possible diagonal relaxation and its optimal cost at the same time, and amounts to solving \eqref{eqn:maxSumSmallest} for a maximum of $\lfloor \log_{2} N \rfloor + 1$ values of $K$.  Since the function $S_{K}$ is concave in $\mbD$ \cite{bv2004} and the constraints in \eqref{eqn:maxSumSmallest} are convex, the maximization in \eqref{eqn:maxSumSmallest} is a convex problem.  Furthermore, \eqref{eqn:maxSumSmallest} can be recast as a standard semidefinite program following \cite{bv2004} by expressing the function $S_{K}$ as the optimal cost of a linear program and then substituting the Lagrangian dual of the linear program.  Thus \eqref{eqn:maxSumSmallest} can be solved efficiently using standard interior-point algorithms.  Further efficiency enhancements can be made as detailed in \cite[Sec.~3.5]{wei2011}.

\subsection{Worst-case instances}
\label{subsec:diagRelaxWorstCase}

As with the continuous relaxation in Sect.~\ref{subsec:linRelaxEx}, we consider extreme instances in which the diagonal relaxation is either a tight approximation to the original problem or an arbitrarily poor one.  It is clear that if $\mbQ$ is already diagonal, the diagonal relaxation and the original problem coincide and the approximation ratio defined in Sect.~\ref{subsec:linRelaxEx} is equal to $1$.  It is shown that the approximation ratio can also equal zero, i.e., the optimal cost of the diagonal relaxation can be zero while the original problem has a non-zero optimal cost.  Based on Fig.~\ref{fig:diagRelax}, the diagonal relaxation is expected to result in a poor approximation when the original ellipsoid $\mcE_{\mbQ}$ is far from coordinate-aligned, thus forcing the coordinate-aligned enclosing ellipsoid $\mcE_{\mbD}$ to be much larger than $\mcE_{\mbQ}$.  This situation is exemplified by the first class of instances in Sect.~\ref{subsec:linRelaxEx} in which $\mcE_{\mbQ}$ is dominated by a single long axis with equal components in all coordinate directions.  To show that the diagonal relaxation has an optimal cost of zero in these instances, we make use of the following lemma.

\begin{lemma}\label{lem:maxSumSmallestLB}
Assume that the vector $\mbc$ has unit-magnitude components.  Then the optimal cost $E_d(K)$ in \eqref{eqn:maxSumSmallest} is bounded from below by $K \lambda_{\min}(\mbQ)$.  This lower bound is tight if the eigenvector $\mbv$ corresponding to $\lambda_{\min}(\mbQ)$ has components of equal magnitude.
\end{lemma}
\begin{proof}
The diagonal matrix $\mbD = \lambda_{\min}(\mbQ) \mbI$ satisfies $\mbD \preceq \mbQ$ and is therefore a feasible solution to \eqref{eqn:maxSumSmallest}.  Hence the corresponding objective value $K \lambda_{\min}(\mbQ)$ (with $c_{n}^{2} = 1$ for all $n$) is a lower bound on $E_d(K)$.  If the eigenvector $\mbv$ has equal-magnitude components and is normalized to have unit 2-norm, then the inequality $\mbD \preceq \mbQ$ implies that 
\begin{equation}\label{eqn:traceUB}
\mbv^T \mbD \mbv = \frac{1}{N} \sum_{n=1}^{N} D_{nn} \leq \mbv^T \mbQ \mbv = \lambda_{\min}(\mbQ)
\end{equation}
for any feasible $\mbD$ in \eqref{eqn:maxSumSmallest}.  The solution $\mbD = \lambda_{\min}(\mbQ) \mbI$ satisfies \eqref{eqn:traceUB} with equality and is therefore an optimal solution to \eqref{eqn:maxSumSmallest} for $K = N$ under the assumptions of the lemma, yielding $E_{d}(N) = N \lambda_{\min}(\mbQ)$.  Using the fact that the mean of the $K$ smallest $D_{nn}$ for $K < N$ is no greater than the mean of all $N$ diagonal entries, it follows from \eqref{eqn:traceUB} that 
\begin{equation}\label{eqn:sumSmallestUB2}
S_{K} \bigl( \{D_{nn}\} \bigr) \leq K \lambda_{\min}(\mbQ), \quad K = 1, 2, \ldots, N-1,
\end{equation}
again for any feasible $\mbD$ in \eqref{eqn:maxSumSmallest}.  Since the solution $\mbD = \lambda_{\min}(\mbQ) \mbI$ also satisfies \eqref{eqn:sumSmallestUB2} with equality, it is an optimal solution to \eqref{eqn:maxSumSmallest} for all $K$ under the assumptions of the lemma and we have $E_{d}(K) = K \lambda_{\min}(\mbQ)$.
\qed\end{proof}

In the first class of instances in Sect.~\ref{subsec:linRelaxEx}, $\mbc = \mbe$, $\lambda_{\min}(\mbQ) = \lambda_1 = 1/N$ and the corresponding eigenvector $\mbv$ has equal-magnitude components.  It follows from Lemma \ref{lem:maxSumSmallestLB} that $E_d(K) = K \lambda_1 = K / N$, which does not exceed $\gamma = 1$ for any $K$.  Hence the optimal cost of the diagonal relaxation is zero while the optimal cost in the unrelaxed problem \eqref{eqn:sparseProb} is $\lfloor N/2 \rfloor$.  This implies that it is not possible to bound the approximation ratio away from zero for all instances of \eqref{eqn:sparseProb}, as with the continuous relaxation.  Furthermore, since the continuous relaxation yields a tight approximation for the same class of instances, neither relaxation strictly dominates the other (diagonal relaxations are clearly dominant in the case of diagonal $\mbQ$).  These conclusions however are based on extreme instances.  It will be seen in Sect.~\ref{sec:numEx} that in more typical instances the diagonal relaxation can offer a significantly better quality of approximation than the continuous relaxation.  In addition, non-trivial lower bounds on the diagonal relaxation approximation ratio can be obtained as in Sect.~\ref{subsec:diagRelaxEig}--\ref{subsec:diagRelaxNAA} when the class of instances of \eqref{eqn:sparseProb} is restricted.

\subsection{Eigenvalue-based approximation guarantees}
\label{subsec:diagRelaxEig}

In this subsection, the quality of approximation of the diagonal relaxation is characterized in terms of the eigenvalues of the matrix $\mbQ$.  The resulting bounds on the approximation ratio are strongest in the case of well-conditioned $\mbQ$, i.e., when the eigenvalues of $\mbQ$ have a low spread.  Geometrically, the well-conditioned case corresponds to a nearly spherical ellipsoid $\mcE_{\mbQ}$, which can be enclosed by a coordinate-aligned ellipsoid $\mcE_{\mbD}$ of comparable size as illustrated in Fig.~\ref{fig:diagRelaxCond}.  Given the close approximation of $\mcE_\mbQ$ by $\mcE_\mbD$ in terms of volume, one would expect a close approximation in terms of the cardinality cost as well.  This geometric intuition is confirmed by the analysis.

\begin{figure}[ht]
\centering
\psfrag{EQ}[][]{$\mcE_\mbQ$}
\psfrag{ED}[][]{$\mcE_\mbD$}
\includegraphics[width=0.6\columnwidth]{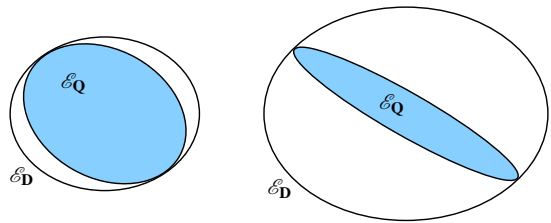}
\caption{Diagonal relaxations for two ellipsoids $\mcE_{\mbQ}$ with contrasting condition numbers.}
\label{fig:diagRelaxCond}
\end{figure}

The results presented in the remainder of the section are more conveniently stated in terms of the number of zero-valued components rather than the number of non-zero components.  Define $K^{\ast}$ to be the maximum number of zero-valued components in \eqref{eqn:sparseProb} and $K_{d}$ to be the maximum number of zero-valued components in the diagonal relaxation of \eqref{eqn:sparseProb}.  The enclosing condition $\mcE_{\mbQ} \subseteq \mcE_{\mbD}$ ensures that $K^{\ast} \leq K_{d}$, and a good approximation corresponds to the ratio $K_{d} / K^{\ast}$ being not much larger than $1$.  It is shown that $K^{\ast}$ and $K_{d}$ can be bounded by the following quantities related to the eigenvalues of $\mbQ$ and its Schur complements:
\begin{subequations}\label{eqn:Kbar}
\begin{align}
\underline{K} &= \max \left\{ K : \lambda_{\max}(\mbQ / \mbQ_{Y(K)Y(K)}) S_{K}(\{ c_{n}^{2} \}) \leq \gamma \right\},\label{eqn:Kunderbar}\\
\overline{K} &= \max \left\{ K : \lambda_{\min}(\mbQ) S_{K}(\{ c_{n}^{2} \}) \leq \gamma \right\},\label{eqn:Koverbar}
\end{align}
\end{subequations}
where $Y(K)$ denotes the index set corresponding to the $N-K$ largest-magnitude components of $\mbc$ (its complement $Z(K)$ corresponds to the $K$ smallest components).  The relationships among $K^{\ast}$, $K_{d}$, $\underline{K}$ and $\overline{K}$ are specified below.

\begin{theorem}\label{thm:diagRelaxEig}
The maximum numbers of zero-valued components in problem \eqref{eqn:sparseProb} and its diagonal relaxation, $K^{\ast}$ and $K_{d}$ respectively, satisfy the ordering $\underline{K} \leq K^{\ast} \leq K_{d} \leq \overline{K}$, where $\underline{K}$ and $\overline{K}$ are defined in \eqref{eqn:Kbar}.  Furthermore, the approximation ratio $K_{d} / K^{\ast}$ is bounded as follows:
\begin{equation}\label{eqn:approxRatioUBEig}
\frac{K_{d}}{K^{\ast}} \leq \frac{\overline{K}}{\underline{K}} \leq 
\frac{\left\lceil (\underline{K} + 1) \lambda_{\max}(\mbQ / \mbQ_{Y(\underline{K}+1)Y(\underline{K}+1)}) / \lambda_{\min}(\mbQ) \right\rceil - 1}{\underline{K}}.
%\approx \frac{\lambda_{\max}(\mbQ / \mbQ_{\mcY(\underline{K}+1)\mcY(\underline{K}+1)})}{\lambda_{\min}(\mbQ)}.
\end{equation}
\end{theorem}
\begin{proof}
The quantity $K^{\ast}$ is equivalently the largest value of $K$ such that condition \eqref{eqn:feasTestK} is satisfied, and hence $K^{\ast}$ can be bounded from below through an upper bound on $E_{0}(K)$ in \eqref{eqn:feasTestK}.  By choosing a specific subset $Z(K)$ corresponding to the $K$ smallest-magnitude components of $\mbc$, we obtain 
\begin{align}
E_{0}(K) = \min_{\abs{Z} = K} \left\{ \mbc_Z^T( \mbQ / \mbQ_{YY} ) \mbc_Z \right\} &\leq \mbc_{Z(K)}^T( \mbQ / \mbQ_{Y(K)Y(K)} ) \mbc_{Z(K)} \notag\\
&\leq \lambda_{\max}( \mbQ / \mbQ_{Y(K)Y(K)} ) S_{K}(\{c_{n}^{2}\}), \label{eqn:E0UBEig}
\end{align}
where the second inequality is due to a property of quadratic forms \cite{hornjohnson1994}.  It follows from \eqref{eqn:E0UBEig} and the definition of $\underline{K}$ in \eqref{eqn:Kunderbar} that $K^{\ast} \geq \underline{K}$.  Similarly, $K_{d}$ is the largest value of $K$ such that $E_{d}(K)$ in \eqref{eqn:maxSumSmallest} is no greater than $\gamma$ and can therefore be bounded from above through a lower bound on $E_{d}(K)$.  Since $\mbD = \lambda_{\min}(\mbQ) \mbI$ is a feasible solution to \eqref{eqn:maxSumSmallest}, we have $E_{d}(K) \geq \lambda_{\min}(\mbQ) S_{K}(\{c_{n}^{2}\})$ and $K_{d} \leq \overline{K}$ from the definition of $\overline{K}$ in \eqref{eqn:Koverbar}.

To obtain the bound on the ratio $\overline{K} / \underline{K}$, we infer from the definition of $\underline{K}$ in \eqref{eqn:Kunderbar} that $\lambda_{\max}(\mbQ / \mbQ_{Y(\underline{K}+1)Y(\underline{K}+1)}) S_{\underline{K}+1}(\{c_{n}^{2}\}) > \gamma$.  The left-hand side of this inequality can be bounded from above as follows:
\begin{equation}\label{eqn:KbarRatioUB}
\lambda_{\max}(\mbQ / \mbQ_{Y(\underline{K}+1)Y(\underline{K}+1)}) S_{\underline{K}+1}(\{c_{n}^{2}\}) 
\leq \lceil k \rceil \lambda_{\min}(\mbQ) \frac{S_{\underline{K}+1}(\{c_{n}^{2}\})}{\underline{K}+1} 
\leq \lambda_{\min}(\mbQ) S_{\lceil k \rceil}(\{c_{n}^{2}\}),
\end{equation}
where $k = (\underline{K} + 1) \lambda_{\max}(\mbQ / \mbQ_{Y(\underline{K}+1)Y(\underline{K}+1)}) / \lambda_{\min}(\mbQ) \geq \underline{K}+1$.  The last inequality in \eqref{eqn:KbarRatioUB} is due to the fact that the mean of the smallest elements in a sequence is non-decreasing when a larger number of elements is included.  From the inequality $\lambda_{\min}(\mbQ) S_{\lceil k \rceil}(\{c_{n}^{2}\}) > \gamma$ and the definition of $\overline{K}$ in \eqref{eqn:Koverbar}, we conclude that $\overline{K} \leq \lceil k \rceil - 1$.  %The approximation to the ratio $(\lceil k \rceil - 1) / \underline{K} \approx \lambda_{\max}(\mbQ / \mbQ_{\mcY_{\underline{K}+1}\mcY_{\underline{K}+1}}) / \lambda_{\min}(\mbQ)$ is justified when $\underline{K}$ is large.
\qed\end{proof}

In the limit of large $\underline{K}$, the bound on the approximation ratio $K_{d} / K^{\ast}$ in Theorem \ref{thm:diagRelaxEig} is approximately equal to the eigenvalue ratio $\lambda_{\max}(\mbQ / \mbQ_{Y(\underline{K}+1)Y(\underline{K}+1)}) / \lambda_{\min}(\mbQ)$, which can be regarded as a type of condition number.  This eigenvalue ratio is in turn bounded from above by the conventional condition number $\kappa(\mbQ) = \lambda_{\max}(\mbQ) / \lambda_{\min}(\mbQ)$ \cite{hornjohnson1994}, thus linking approximation quality in terms of cardinality to the geometric approximation quality illustrated in Fig.~\ref{fig:diagRelaxCond}.

Theorem \ref{thm:diagRelaxEig} can be strengthened somewhat by exploiting an invariance property of problem \eqref{eqn:sparseProb} and its diagonal relaxation.  It is straightforward to see that the optimal cost in \eqref{eqn:sparseProb} (and hence $K^{\ast}$) is invariant to diagonal scaling transformations of the feasible set, i.e., transformations parameterized by an invertible diagonal matrix $\mbS$ mapping $\mbc$ to $\mbS \mbc$ and $\mbQ$ to $\mbS^{-1} \mbQ \mbS^{-1}$.  Likewise, the optimal cost $E_{d}(K)$ in \eqref{eqn:maxSumSmallest} can be shown to be invariant to the same transformations, and thus $K_{d}$ is invariant \cite{wei2011}.  By generalizing the definitions of $\underline{K}$ and $\overline{K}$, Theorem \ref{thm:diagRelaxEig} can be generalized as follows:
\begin{corollary}\label{cor:diagRelaxEig}
For any invertible diagonal matrix $\mbS$, define $Y_{\mbS}(K)$ to be the index set corresponding to the $N-K$ largest $S_{nn} c_{n}^{2}$ and 
%
%\begin{subequations}\label{eqn:Kbar}
\begin{align*}
\underline{K}_{\mbS} &= \max \left\{ K : \lambda_{\max}((\mbS^{-1} \mbQ \mbS^{-1}) / (\mbS^{-1} \mbQ \mbS^{-1})_{Y_{\mbS}(K)Y_{\mbS}(K)}) S_{K}(\{ S_{nn} c_{n}^{2} \}) \leq \gamma \right\},\\
\overline{K}_{\mbS} &= \max \left\{ K : \lambda_{\min}(\mbS^{-1} \mbQ \mbS^{-1}) S_{K}(\{ S_{nn} c_{n}^{2} \}) \leq \gamma \right\}.
\end{align*}
%\end{subequations}
%
Then Theorem \ref{thm:diagRelaxEig} holds with $\mbQ$, $\underline{K}$, $\overline{K}$, and $Y(K)$ replaced by $\mbS^{-1} \mbQ \mbS^{-1}$, $\underline{K}_{\mbS}$, $\overline{K}_{\mbS}$, and $Y_{\mbS}(K)$ respectively.
\end{corollary}

\noindent The scaling matrix $\mbS$ can be chosen to minimize the eigenvalue ratio in Theorem \ref{thm:diagRelaxEig}, i.e., as a type of optimal diagonal preconditioner for $\mbQ$, thus minimizing the bound on the approximation ratio.

The bounds in Theorem \ref{thm:diagRelaxEig} are essentially tight.  Specifically, it is shown that for $N \geq 5$, the inequalities $\underline{K} \leq K^{\ast}$ and $K_{d} \leq \overline{K}$ can be simultaneously tight so that the left-hand inequality in \eqref{eqn:approxRatioUBEig} is met with equality, while the right-hand inequality reduces to $\overline{K} / \underline{K} \leq (\overline{K}+1) / \underline{K}$ and is asymptotically tight as $\underline{K} \to \infty$.  We consider again the first class of instances constructed in Sect.~\ref{subsec:linRelaxEx} in which $\mbc = \mbe$, $\gamma = 1$, and the eigenvector $\mbv$ corresponding to the smallest eigenvalue of $\mbQ$ has $\lceil N/2 \rceil$ components equal to $+1/\sqrt{N}$ and $\lfloor N/2 \rfloor$ components equal to $-1/\sqrt{N}$.  We keep $\lambda_{1} = 1/N$ and change $\lambda_{2}$ to $\lambda_{2} = 1/(2 \lceil N/2 \rceil - \lfloor \sqrt{N} \rfloor - 1)$.  Given these choices, \eqref{eqn:Koverbar} yields $\overline{K} = N = 1/\lambda_{1}$, while from \eqref{eqn:QschurEx} we have $\lambda_{\max}(\mbQ / \mbQ_{Y(K)Y(K)}) = \lambda_{2}$ and hence $\underline{K} = 1/\lambda_{2} = 2 \lceil N/2 \rceil - \lfloor \sqrt{N} \rfloor - 1$ from \eqref{eqn:Kunderbar}.  It can then be verified through substitution that the rightmost quantity in \eqref{eqn:approxRatioUBEig} is equal to $(\overline{K} + 1) / \underline{K}$ for $N \geq 5$ as claimed. 
%
%\[
%\frac{\left\lceil (\underline{K} + 1) \lambda_{\max}(\mbQ / \mbQ_{\mcY(\underline{K}+1)\mcY(\underline{K}+1)}) / \lambda_{\min}(\mbQ) \right\rceil - 1}{\underline{K}} = \frac{\overline{K}+1}{\underline{K}}
%\]
%
Furthermore, the construction satisfies the assumptions of Lemma \ref{lem:maxSumSmallestLB} and thus $E_{d}(K) = K \lambda_{\min}(\mbQ) = K/N$, from which it follows that $K_{d} = N = \overline{K}$.

It remains to show that $\underline{K} = K^{\ast}$ for this class of instances.  This is equivalent to showing that condition \eqref{eqn:feasTestK} is violated for $K = \underline{K} + 1$.  Substituting \eqref{eqn:QschurEx} and the chosen parameter values into \eqref{eqn:feasTestK} and performing some simplifications, the required condition $E_{0}(\underline{K} + 1) > \gamma$ is equivalent to
\begin{equation}\label{eqn:feasTestDiagRelaxEig}
(\underline{K} + 1)(\kappa - 1) \max_{\abs{Z} = \underline{K} + 1} (\mbe^{T} \hat{\mbv}_{Z})^{2} < (\underline{K} + 1)(\kappa - 1) + N,
\end{equation}
where $\kappa = \lambda_{2} / \lambda_{1} = \overline{K} / \underline{K}$.  As was the case in \eqref{eqn:feasTestLinRelaxBestCase}, the maximum in \eqref{eqn:feasTestDiagRelaxEig} is achieved by including in $Z$ all $\lceil N/2 \rceil$ positive components of $\mbv$, with the remaining components being negative.  Noting that $\underline{K} + 1 = 2 \lceil N/2 \rceil - \lfloor \sqrt{N} \rfloor \geq \lceil N/2 \rceil$ for $N \geq 5$, the maximum value can be seen to be $\lfloor \sqrt{N} \rfloor^{2} / (\underline{K} + 1)$.  Condition \eqref{eqn:feasTestDiagRelaxEig} then becomes 
\[
\left(2 \lceil N/2 \rceil - \lfloor \sqrt{N} \rfloor^{2}\right)(\kappa - 1) + N - \lfloor \sqrt{N} \rfloor (\kappa - 1) > 0,
\]
which is true given that $1 < \kappa \leq 2$ for $N \geq 5$.

Theorem \ref{thm:diagRelaxEig} and Corollary \ref{cor:diagRelaxEig} characterize the approximation quality of the diagonal relaxation in terms of extreme eigenvalues, specifically the smallest eigenvalue of $\mbQ$ and the largest eigenvalue of a Schur complement of $\mbQ$.  A second characterization involving intermediate eigenvalues can be obtained under the stochastic assumption that the eigenvectors of $\mbQ$ are chosen as an orthonormal set uniformly at random from the unit sphere.  This assumption allows the bound on $E_{0}(K)$ in \eqref{eqn:E0UBEig} to be improved, essentially replacing the largest eigenvalue of $\mbQ / \mbQ_{Y(\underline{K}+1)Y(\underline{K}+1)}$ with the mean eigenvalue of $\mbQ$, $\lambdab(\mbQ) = \frac{1}{N} \sum_{n=1}^{N} \lambda_{n}(\mbQ)$.  By retaining the other elements in the proof of Theorem \ref{thm:diagRelaxEig}, we obtain the following bound on the approximation ratio, which holds with high probability as $N$ becomes large.

\begin{theorem}\label{thm:diagRelaxEigRand}
Let the matrix $\mbV$ of eigenvectors of $\mbQ$ be drawn uniformly at random from the set of $N\times N$ orthogonal matrices.  Then the approximation ratio $K_{d} / K^{\ast}$ is bounded from above by 
\[
\frac{\lceil (\underline{K}+1) (1+\epsilon) \lambdab(\mbQ) / \lambda_{\min}(\mbQ) \rceil - 1}{\underline{K}}
\]
with probability at least 
\begin{equation}\label{eqn:approxRatioUBProb}
\begin{cases}
1 - \exp\left( -\frac{N}{8} \frac{\epsilon^{2} \lambdab(\mbQ)^{2}}{\epsilon^{2} \lambdab(\mbQ)^{2} + \var(\lambda(\mbQ))} \right), & \epsilon \in (0, \epsilon_{\max}) \backslash \mathcal{I},\\
1 - \exp\left( -\frac{N}{8} \frac{\epsilon \lambdab(\mbQ)}{\epsilon \lambdab(\mbQ) + (\lambda_{\max}(\mbQ) - \lambdab(\mbQ))} \right), & \epsilon \in \mathcal{I},\\
1, & \epsilon \geq \epsilon_{\max},
\end{cases}
\end{equation}
where $\var(\lambda(\mbQ)) = \frac{1}{N} \sum_{n=1}^{N} (\lambda_{n}(\mbQ) - \lambdab(\mbQ))^{2}$ is the variance of the eigenvalues of $\mbQ$, $\epsilon_{\max} = \lambda_{\max}(\mbQ) / \lambdab(\mbQ) - 1$,
\begin{equation}\label{eqn:epsInterval}
\mathcal{I} = \begin{cases}
(\epsilon_{-}, \epsilon_{+}), &(\lambda_{\max}(\mbQ) - \lambdab(\mbQ))^{2} > 8\var(\lambda(\mbQ)),\\
\emptyset, &(\lambda_{\max}(\mbQ) - \lambdab(\mbQ))^{2} \leq 8\var(\lambda(\mbQ)),
\end{cases}
\end{equation}
and
\[
\epsilon_{\pm} = \frac{1}{4} \left( \epsilon_{\max} \pm \sqrt{\epsilon_{\max}^{2} - 8 \frac{\var(\lambda(\mbQ))}{\lambdab(\mbQ)^{2}} } \right).
\]
\end{theorem}
\begin{proof}
As noted above, it suffices to replace the bound in \eqref{eqn:E0UBEig} with 
\begin{equation}\label{eqn:E0UBEigRand}
E_{0}(K) \leq (1+\epsilon) \lambdab(\mbQ) S_{K}(\{c_{n}^{2}\})
\end{equation}
and show that \eqref{eqn:E0UBEigRand} holds with the probabilities indicated in the theorem statement.  The remainder of the proof proceeds as in Theorem \ref{thm:diagRelaxEig}.  First note that for $\epsilon \geq \epsilon_{\max}$, \eqref{eqn:E0UBEigRand} is implied by \eqref{eqn:E0UBEig} and is therefore true with probability $1$.  For $\epsilon \in (0,\epsilon_{\max})$, we use an upper bound on $E_{0}(K)$ to bound the probability that \eqref{eqn:E0UBEigRand} is violated.  Choosing the same subset $Z(K)$ as in Theorem \ref{thm:diagRelaxEig} and using the definition of the Schur complement, we have  
\begin{equation}\label{eqn:E0UBEigRand2}
E_{0}(K) \leq \mbc_{Z(K)}^{T} \mbQ_{Z(K)Z(K)} \mbc_{Z(K)}
= \mbct^{T} \mbLambda \mbct, \qquad \mbct = \mbV^{T} \begin{bmatrix} \mbc_{Z(K)} \\ \0 \end{bmatrix},
\end{equation}
where $\mbLambda$ is the diagonal matrix of eigenvalues of $\mbQ$.  The assumption on $\mbV$ implies that $\mbct$ is distributed uniformly over the sphere of radius $\sqrt{S_{K}(\{ c_{n}^{2} \})}$ centered at the origin.  Hence the quantity $\mbct^{T} \mbLambda \mbct$ can be equivalently expressed as $S_{K}(\{ c_{n}^{2} \}) (\mbz^{T} \mbLambda \mbz / \mbz^{T} \mbz)$, where the components of $\mbz$ are independent standard normal random variables.  

We now bound the probability that $\mbct^{T} \mbLambda \mbct > (1+\epsilon) \lambdab(\mbQ) S_{K}(\{ c_{n}^{2} \})$, which in turn bounds the probability that \eqref{eqn:E0UBEigRand} is not satisfied.  The event in question can be rewritten as 
\[
S = \sum_{n=1}^{N} \left[\lambda_{n}(\mbQ) - (1+\epsilon) \lambdab(\mbQ) \right] z_{n}^{2} \equiv \sum_{n=1}^{N} \delta_{n} z_{n}^{2} > 0.
\]
It can be seen that the expected value of $S$ is equal to $-\epsilon N \lambdab(\mbQ)$, and hence we are bounding the probability that a linear combination of independent chi-squared random variables exceeds its mean by $\epsilon N \lambdab(\mbQ)$.  A straightforward application of the Chernoff bound \cite{chernoff1952} yields 
\[
\log \Pr(S > 0) \leq \min_{0 \leq t < 1/(2\delta_{\max})} -\frac{1}{2} \sum_{n=1}^{N} \log(1 - 2\delta_{n} t),
\]
where $\delta_{\max} = \lambda_{\max}(\mbQ) - (1+\epsilon) \lambdab(\mbQ)$.  To derive a closed-form expression for the Chernoff exponent, the function $-(1/2) \log(1 - 2\delta_{n}t)$ is bounded from above by the quadratic function $2\delta_{n}^{2} t^{2} + \delta_{n} t$ over the interval $[0, 1/(4\delta_{\max})]$ (this upper bound can be verified by comparing derivatives over $[0, 1/(4\delta_{\max})]$).  It follows that 
\begin{equation}\label{eqn:chernoff}
\log \Pr(S > 0) \leq N \min_{0 \leq t \leq 1/(4\delta_{\max})} 2 \left( \var(\lambda(\mbQ)) + \epsilon^{2} \lambdab(\mbQ)^{2} \right) t^{2} - \epsilon \lambdab(\mbQ) t,
\end{equation}
using the definition of $\var(\lambda(\mbQ))$.  We consider the two cases in which the unconstrained minimizer $t^{\ast} = (1/4) \epsilon \lambdab(\mbQ) / (\var(\lambda(\mbQ)) + \epsilon^{2} \lambdab(\mbQ)^{2})$ is either less than or greater than $1/(4\delta_{\max})$.  These correspond to the first two cases in \eqref{eqn:approxRatioUBProb}.  In the first case, substituting $t = t^{\ast}$ into \eqref{eqn:chernoff} yields the exponent in \eqref{eqn:approxRatioUBProb} directly, while in the second case, the exponent in \eqref{eqn:approxRatioUBProb} results from substituting $t = 1/(4\delta_{\max})$ in \eqref{eqn:chernoff} and then using the assumed inequality $t^{\ast} > 1/(4\delta_{\max})$.  Solving the boundary condition $t^{\ast} = 1/(4\delta_{\max})$ for $\epsilon$ yields the expression in \eqref{eqn:epsInterval} for the interval $\mathcal{I}$.
\qed\end{proof}

Theorem \ref{thm:diagRelaxEigRand} can be significantly less conservative than Theorem \ref{thm:diagRelaxEig}, in particular when most of the eigenvalues are small and comparable so that the mean eigenvalue of $\mbQ$ is much closer to the minimum eigenvalue than to the maximum eigenvalue.  This preference for eigenvalue distributions weighted toward small values is seen in the numerical results in Sect.~\ref{sec:numEx}.  Furthermore, it agrees with the following geometric intuition: Assuming that the ellipsoid $\mcE_{\mbQ}$ is not close to spherical ($\kappa(\mbQ)$ is large), it is preferable for most of the ellipsoid axes to be comparatively long (corresponding to small eigenvalues) and of the same order.  Such an ellipsoid tends to require a smaller coordinate-aligned enclosing ellipsoid, and consequently the diagonal relaxation tends to be a better approximation.  For example, in three dimensions, a severely oblate spheroid can be enclosed on average in a smaller coordinate-aligned ellipsoid than an equally severely prolate spheroid.  Note also that the exponents in \eqref{eqn:approxRatioUBProb} depend on the eigenvalue distribution and are larger (i.e., the decay is sharper) when the spread of the eigenvalues is small as measured by $\var(\lambda(\mbQ))$ or $\lambda_{\max}(\mbQ) - \lambdab(\mbQ)$.

\subsection{The diagonally dominant case}
\label{subsec:diagRelaxDiagDom}

We now consider the case in which the matrix $\mbQ$ is diagonally dominant, specifically in the sense that
\begin{equation}\label{eqn:QdiagDom}
\max_{m} \sum_{n\neq m} \frac{\abs{Q_{mn}}}{\sqrt{Q_{mm} Q_{nn}}} < 1,
\end{equation}
i.e., the absolute sum of the normalized off-diagonal entries in any row or column is small.  It is expected in this case that the original problem \eqref{eqn:sparseProb} can be well-approximated by its diagonal relaxation, and that the quality of approximation depends on the degree of diagonal dominance.  Indeed, it can be shown that the maximum numbers of zero-valued components in \eqref{eqn:sparseProb} and its diagonal relaxation, $K^{\ast}$ and $K_{d}$ respectively, are bounded by the following quantities related to diagonal dominance:
\begin{subequations}\label{eqn:KbarDiagDom}
\begin{align}
\underline{K}_{\dd} &= \max \left\{ K : \left( 1 + \max_{m\in Z_{\dd}(K)} \sum_{\substack{n\in Z_{\dd}(K) \\ n\neq m}} \frac{\abs{Q_{mn}}}{\sqrt{Q_{mm} Q_{nn}}} \right) S_{K}(\{Q_{nn} c_{n}^{2}\}) \leq \gamma \right\}, \label{eqn:KunderbarDiagDom}\\
\overline{K}_{\dd} &= \max \left\{ K : \left( 1 - \max_{m} \sum_{n\neq m} \frac{\abs{Q_{mn}}}{\sqrt{Q_{mm} Q_{nn}}} \right) S_{K}(\{Q_{nn} c_{n}^{2}\}) \leq \gamma \right\}, \label{eqn:KoverbarDiagDom}
\end{align}
\end{subequations}
where $Z_{\dd}(K)$ in \eqref{eqn:KunderbarDiagDom} denotes the index set corresponding to the $K$ smallest $Q_{nn} c_{n}^{2}$.  A bound on the approximation ratio $K_{d} / K^{\ast}$ follows.

\begin{theorem}\label{thm:diagRelaxDiagDom}
Assume that the matrix $\mbQ$ is diagonally dominant in the sense of \eqref{eqn:QdiagDom}.  Then the maximum numbers of zero-valued components in problem \eqref{eqn:sparseProb} and its diagonal relaxation, $K^{\ast}$ and $K_{d}$ respectively, satisfy the ordering $\underline{K}_{\dd} \leq K^{\ast} \leq K_{d} \leq \overline{K}_{\dd}$, where $\underline{K}_{\dd}$ and $\overline{K}_{\dd}$ are defined in \eqref{eqn:KbarDiagDom}.  The approximation ratio $K_{d} / K^{\ast}$ is bounded as follows:
\begin{equation}\label{eqn:approxRatioUBDiagDom}
\frac{K_{d}}{K^{\ast}} \leq \frac{\overline{K}_{\dd}}{\underline{K}_{\dd}} \leq 
\frac{\left\lceil (\underline{K}_{\dd} + 1) r_{\dd} \right\rceil - 1}{\underline{K}_{\dd}},
\end{equation}
where
\[
r_{\dd} = \left( 1 + \max_{m\in Z_{\dd}(\underline{K}_{\dd}+1)} \sum_{\substack{n\in Z_{\dd}(\underline{K}_{\dd}+1) \\ n \neq m}} \frac{\abs{Q_{mn}}}{\sqrt{Q_{mm} Q_{nn}}} \right) \left/ 
\left( 1 - \max_m \sum_{n \neq m} \frac{\abs{Q_{mn}}}{\sqrt{Q_{mm} Q_{nn}}} \right) \right..
\]
\end{theorem}

The ratio $r_{\dd}$ in Theorem \ref{thm:diagRelaxDiagDom} plays the same role as the eigenvalue ratio 
%$\lambda_{\max}(\mbQ / \mbQ_{\mcY(\underline{K}+1)\mcY(\underline{K}+1)}) / \lambda_{\min}(\mbQ)$
in Theorem \ref{thm:diagRelaxEig}.  As $\mbQ$ becomes more diagonally dominant, $r_{\dd}$ approaches $1$ from above.  Unlike with Theorem \ref{thm:diagRelaxEig}, there is no benefit to allowing diagonal scaling transformations because the measure of diagonal dominance used here remains unchanged when $\mbQ$ is replaced by $\mbS^{-1} \mbQ \mbS^{-1}$.

To prove the inequality $K_{d} \leq \overline{K}_{\dd}$, we use the following lemma, which specifies the optimal cost of \eqref{eqn:maxSumSmallest} under the additional constraint that $\mbD$ is a multiple of a fixed diagonal matrix.
\begin{lemma}\label{lem:maxSumSmallestLBD0}
For any positive definite diagonal matrix $\mbD_{0}$, the optimal cost $E_{d}(K)$ in \eqref{eqn:maxSumSmallest} is bounded from below by $\lambda_{\min}(\mbD_{0}^{-1/2} \mbQ \mbD_{0}^{-1/2}) S_{K}(\{ (\mbD_{0})_{nn} c_{n}^{2} \})$.
\end{lemma}
\begin{proof}
We restrict $\mbD$ in \eqref{eqn:maxSumSmallest} to be a multiple of $\mbD_{0}$, thus obtaining a lower bound on $E_{d}(K)$.  With $\mbD = \alpha \mbD_{0}$, \eqref{eqn:maxSumSmallest} reduces to
\[
\max_{\alpha} \quad \alpha S_{K}(\{ (\mbD_{0})_{nn} c_{n}^{2} \}) \qquad 
\text{s.t.} \qquad \0 \preceq \alpha\mbD_{0} \preceq \mbQ.
\]
Since $\mbD_{0}$ is invertible, the constraint can be rewritten as $\0 \preceq \alpha\mbI \preceq \mbD_{0}^{-1/2} \mbQ \mbD_{0}^{-1/2}$, from which it follows that $\alpha$ should be chosen as the smallest eigenvalue of $\mbD_{0}^{-1/2} \mbQ \mbD_{0}^{-1/2}$.
\qed\end{proof}

We now proceed with the proof of Theorem \ref{thm:diagRelaxDiagDom}.

\begin{proof}[Theorem \ref{thm:diagRelaxDiagDom}]
%{\em Proof of Theorem \ref{thm:diagRelaxDiagDom}}.
To prove that $K_{d} \leq \overline{K}_{\dd}$, we let $\mbD_{0} = \Diag(\mbQ)$ in Lemma \ref{lem:maxSumSmallestLBD0}, where $\Diag(\mbQ)$ denotes a diagonal matrix with the same diagonal entries as $\mbQ$.  Using the Gershgorin circle theorem \cite{hornjohnson1994} to bound the smallest eigenvalue of $\mbQt = \Diag(\mbQ)^{-1/2} \mbQ \Diag(\mbQ)^{-1/2}$, we then obtain 
\[
E_{d}(K) \geq \left( 1 - \max_{m} \sum_{n\neq m} \frac{\abs{Q_{mn}}}{\sqrt{Q_{mm} Q_{nn}}} \right) S_{K}(\{Q_{nn} c_{n}^{2}\}),
\]
from which we infer that $K_{d} \leq \overline{K}_{\dd}$ based on \eqref{eqn:KoverbarDiagDom}.

To prove that $K^{\ast} \geq \underline{K}_{\dd} $, the quantity $E_{0}(K)$ in \eqref{eqn:feasTestK} is bounded from above as follows, starting with the specific choice of subset $Z = Z_{\dd}(K)$:
\begin{align*}
E_{0}(K) &\leq \mbc_{Z_{\dd}(K)}^T( \mbQ / \mbQ_{Y_{\dd}(K)Y_{\dd}(K)} ) \mbc_{Z_{\dd}(K)}\\
&\leq \mbc_{Z_{\dd}(K)}^T \mbQ_{Z_{\dd}(K)Z_{\dd}(K)} \mbc_{Z_{\dd}(K)}\\
&= (\Diag(\mbQ)^{1/2} \mbc)_{Z_{\dd}(K)}^{T} 
\mbQt_{Z_{\dd}(K)Z_{\dd}(K)} 
(\Diag(\mbQ)^{1/2} \mbc)_{Z_{\dd}(K)}\\
&\leq \lambda_{\max}( \mbQt_{Z_{\dd}(K)Z_{\dd}(K)} ) S_{K}(\{Q_{nn} c_{n}^{2}\})\\
&\leq \left( 1 + \max_{m\in Z_{\dd}(K)} \sum_{\substack{n\in Z_{\dd}(K) \\ n\neq m}} \frac{\abs{Q_{mn}}}{\sqrt{Q_{mm} Q_{nn}}} \right) S_{K}(\{Q_{nn} c_{n}^{2}\}).
\end{align*}
The second line follows from the definition of the Schur complement, the third from a rescaling, the fourth from eigenvalue properties and the definition of $Z_{\dd}(K)$, and the last from the Gershgorin circle theorem.  Comparing with \eqref{eqn:KunderbarDiagDom}, we conclude that $K^{\ast} \geq \underline{K}_{\dd}$.  The proof of the bound on $\overline{K}_{\dd} / \underline{K}_{\dd}$ is similar to that in Theorem \ref{thm:diagRelaxEig}.
%\end{proof}
\qed\end{proof}

As with Theorem \ref{thm:diagRelaxEig}, there exist instances for which the left-hand bound in \eqref{eqn:approxRatioUBDiagDom} is tight and the right-hand bound is asymptotically tight.  We consider the same class of instances as in Sect.~\ref{subsec:linRelaxEx} with $\mbc = \mbe$, $\gamma = 1$, and $\mbv$ having $\lceil N/2 \rceil$ components equal to $+1/\sqrt{N}$ and $\lfloor N/2 \rfloor$ components equal to $-1/\sqrt{N}$.  From \eqref{eqn:Qex} we obtain $Q_{nn} = (N-1)\lambda_{2}/N + \lambda_{1}/N$ for all $n$ and $\abs{Q_{mn}} = (\lambda_{2} - \lambda_{1})/N$ for all $m \neq n$, from which it follows that 
\begin{align*}
\left( 1 - \max_{m} \sum_{n\neq m} \frac{\abs{Q_{mn}}}{\sqrt{Q_{mm} Q_{nn}}} \right) S_{K}(\{Q_{nn} c_{n}^{2}\}) &= K \lambda_{1},\\
\left( 1 + \max_{m\in Z} \sum_{\substack{n\in Z \\ n\neq m}} \frac{\abs{Q_{mn}}}{\sqrt{Q_{mm} Q_{nn}}} \right) S_{K}(\{Q_{nn} c_{n}^{2}\}) 
&= K \left( \lambda_{2} + \frac{K-2}{N} (\lambda_{2} - \lambda_{1}) \right)
\end{align*}
for any $Z$ of cardinality $K$.  Choosing $\lambda_{1} = 1/N$ and $\lambda_{2} = 1/N + 1/((N-1)(2N-3))$, some straightforward calculations yield $\underline{K}_{\dd} = N-1$ and $\overline{K}_{\dd} = N$ from \eqref{eqn:KbarDiagDom}, and $r_{\dd} = 1 + 2/(2N-3)$ for the ratio defined in Theorem \ref{thm:diagRelaxDiagDom}.  It can then be seen that the right-hand inequality in \eqref{eqn:approxRatioUBDiagDom} reads $N/(N-1) \leq (N+1)/(N-1)$ for $N \geq 3$, which is asymptotically tight as $N \to \infty$.

To show that the left-hand inequality in \eqref{eqn:approxRatioUBDiagDom} is tight, we note that the construction satisfies the assumptions of Lemma \ref{lem:maxSumSmallestLB} so we again have $E_{d}(K) = K\lambda_{\min}(\mbQ) = K/N$ and $K_{d} = N = \overline{K}_{\dd}$.  The remaining required equality $K^{\ast} = \underline{K}_{\dd} = N-1$ is equivalent to the all-zero solution being infeasible for \eqref{eqn:sparseProb}, i.e., $\mbc^{T} \mbQ \mbc > \gamma = 1$.  Using \eqref{eqn:Qex} and substituting the selected parameter values, we find $\mbc^{T} \mbQ \mbc = 1 + N/((N-1)(2N-3)) > 1$ for $N$ even and $\mbc^{T} \mbQ \mbc = 1 + (N+1)/(N(2N-3)) > 1$ for $N$ odd, completing the demonstration.

\subsection{The nearly coordinate-aligned case}
\label{subsec:diagRelaxNAA}

A geometric analogue to diagonal dominance is the case in which the axes of the ellipsoid $\mcE_{\mbQ}$ are nearly aligned with the coordinate axes.  Algebraically, this corresponds to the eigenvectors of $\mbQ$ being close to the standard basis vectors.  We assume that $\mbQ$ is diagonalized as $\mbQ = \mbV \mbLambda \mbV^{T}$, where the eigenvalues $\lambda_{n}(\mbQ)$ and the eigenvector matrix $\mbV$ are ordered in such a way that $\mbDelta = \mbV - \mbI$ is small, specifically in the sense that its spectral radius $\rho(\mbDelta)$ satisfies $\kappa(\mbQ) \rho(\mbDelta) < 1$.  It is expected in this case that the diagonal relaxation would give a better approximation for smaller $\mbDelta$, i.e., for closer alignments.  Following the approach in Sect.~\ref{subsec:diagRelaxEig}--\ref{subsec:diagRelaxDiagDom}, it is shown that $K^{\ast}$ and $K_{d}$ may be bounded by 
\begin{subequations}\label{eqn:KbarNAA}
\begin{align}
\underline{K}_{\na} &= \max \left\{ K : ( 1 + \kappa(\mbQ) (\rho(\mbDelta) + \rho(\mbDelta)^{2}) ) S_{K}(\{\lambda_{n}(\mbQ) c_{n}^{2}\}) \leq \gamma \right\}, \label{eqn:KunderbarNAA}\\
\overline{K}_{\na} &= \max \left\{ K : ( 1 - \kappa(\mbQ) \rho(\mbDelta) ) S_{K}(\{\lambda_{n}(\mbQ) c_{n}^{2}\}) \leq \gamma \right\}. \label{eqn:KoverbarNAA}
\end{align}
\end{subequations}
%
%which depend on $\mbDelta$.  
The approximation ratio $K_{d} / K^{\ast}$ may be bounded accordingly.

\begin{theorem}\label{thm:diagRelaxNAA}
Assume that the matrix $\mbQ$ can be diagonalized as $\mbQ = (\mbI + \mbDelta) \mbLambda (\mbI + \mbDelta)^{T}$, where $\mbDelta$ is such that $\kappa(\mbQ) \rho(\mbDelta) < 1$.  Then the maximum numbers of zero-valued components in \eqref{eqn:sparseProb} and its diagonal relaxation, $K^{\ast}$ and $K_{d}$ respectively, satisfy the ordering $\underline{K}_{\na} \leq K^{\ast} \leq K_{d} \leq \overline{K}_{\na}$, where $\underline{K}_{\na}$ and $\overline{K}_{\na}$ are defined in \eqref{eqn:KbarNAA}.  The approximation ratio $K_{d} / K^{\ast}$ is bounded as follows:
\begin{equation}\label{eqn:approxRatioUBNAA}
\frac{K_{d}}{K^{\ast}} \leq \frac{\overline{K}_{\na}}{\underline{K}_{\na}} \leq 
\frac{\left\lceil (\underline{K}_{\na} + 1) r_{\na} \right\rceil - 1}{\underline{K}_{\na}},
\end{equation}
where
\[
r_{\na} = \frac{1 + \kappa(\mbQ) (\rho(\mbDelta) + \rho(\mbDelta)^{2})}{1 - \kappa(\mbQ) \rho(\mbDelta)}.
\]
\end{theorem}

Theorem \ref{thm:diagRelaxNAA} characterizes the quality of approximation in terms of the ratio $r_{\na}$.  As $\mbDelta$ approaches $\0$, $r_{\na}$ approaches $1$ as expected.  Similar to Theorem \ref{thm:diagRelaxEig}, Theorem \ref{thm:diagRelaxNAA} may be strengthened using diagonal scaling transformations since both $\rho(\mbDelta)$ and the condition number $\kappa(\mbQ)$ may decrease as $\mbQ$ is transformed into $\mbS^{-1} \mbQ \mbS^{-1}$ for different choices of $\mbS$.  The dependence on the condition number can be explained geometrically as illustrated in Fig.~\ref{fig:NAAkappa}.  On the left, the original ellipsoid $\mcE_{\mbQ}$ is both nearly coordinate-aligned and nearly spherical (i.e., $\kappa(\mbQ)$ is close to $1$), and can therefore be enclosed by a coordinate-aligned ellipsoid that is only slightly larger.  Indeed in the limit $\kappa(\mbQ) = 1$, $\mcE_{\mbQ}$ is spherical and thus already coordinate-aligned, and the eigenvector matrix $\mbV$ can be chosen equal to $\mbI$ resulting in $\mbDelta = \0$.  On the other hand, if $\kappa(\mbQ)$ is large, even a small misalignment between the ellipsoid and coordinate axes results in a much larger enclosing ellipsoid, as seen on the right in Fig.~\ref{fig:NAAkappa}.

\begin{figure}[ht]
\centering
\psfrag{t}[][]{$\theta$}
\includegraphics[width=0.7\columnwidth]{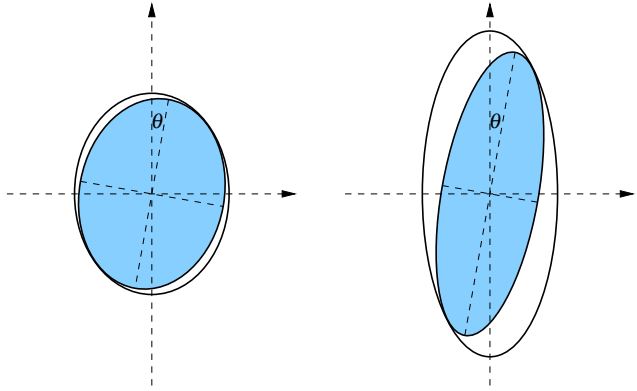}
\caption{The effect of the condition number $\kappa(\mbQ)$ on the approximation quality in the nearly coordinate-aligned case.  For the same angular offset $\theta$ between the axes of the original ellipsoid and the coordinate axes, the coordinate-aligned enclosing ellipsoid on the right is comparatively larger.}
\label{fig:NAAkappa}
\end{figure}

In the proof of Theorem \ref{thm:diagRelaxNAA} below, we make reference to the scaled matrix $\mbLambda^{-1/2} \mbQ \mbLambda^{-1/2}$.  When $\mbDelta$ is small, $\mbLambda^{-1/2} \mbQ \mbLambda^{-1/2}$ is close to the identity matrix and the deviation of its eigenvalues from $1$ is specified by the following lemma.
\begin{lemma}\label{lem:diagRelaxNAA}
Assume that the matrix $\mbQ$ can be diagonalized as $\mbQ = (\mbI + \mbDelta) \mbLambda (\mbI + \mbDelta)^{T}$, where $\mbDelta$ is such that $\kappa(\mbQ) \rho(\mbDelta) < 1$.  Then 
\begin{align*}
\lambda_{\min}(\mbLambda^{-1/2} \mbQ \mbLambda^{-1/2}) &\geq 1 - \kappa(\mbQ) \rho(\mbDelta),\\
\lambda_{\max}(\mbLambda^{-1/2} \mbQ \mbLambda^{-1/2}) &\leq 1 + \kappa(\mbQ) (\rho(\mbDelta) + \rho(\mbDelta)^{2}).
\end{align*}
\end{lemma}
\begin{proof}
Expanding $\mbLambda^{-1/2} \mbQ \mbLambda^{-1/2}$ yields $\mbI + \mbDeltat + \mbDeltat^{T} + \mbDeltat \mbDeltat^{T}$, where $\mbDeltat = \mbLambda^{-1/2} \mbDelta \mbLambda^{1/2}$.  The eigenvalues of $\mbLambda^{-1/2} \mbQ \mbLambda^{-1/2}$ can then be bounded by 
\begin{subequations}
\begin{align}
\lambda_{\min}(\mbLambda^{-1/2} \mbQ \mbLambda^{-1/2}) &\geq 1 + \lambda_{\min}(\mbDeltat + \mbDeltat^{T}),\label{eqn:QtildeEigMin}\\
\lambda_{\max}(\mbLambda^{-1/2} \mbQ \mbLambda^{-1/2}) &\leq 1 + \lambda_{\max}(\mbDeltat + \mbDeltat^{T}) + \lambda_{\max}(\mbDeltat \mbDeltat^{T}),\label{eqn:QtildeEigMax}
\end{align}
\end{subequations}
noting that $\mbDeltat \mbDeltat^{T}$ is positive semidefinite in \eqref{eqn:QtildeEigMin}.  The rightmost term in \eqref{eqn:QtildeEigMax} can be bounded using the sub-multiplicative property of the spectral norm \cite{hornjohnson1994}:
\[
\lambda_{\max}(\mbDeltat \mbDeltat^{T}) = \norm[2]{\mbDeltat^{T}}^{2}
\leq \norm[2]{\mbLambda^{1/2}}^{2} \norm[2]{\mbDelta}^{2} \norm[2]{\mbLambda^{-1/2}}^{2}
= \lambda_{\max}(\mbQ) \rho(\mbDelta)^{2} \lambda_{\min}^{-1}(\mbQ) = \kappa(\mbQ) \rho(\mbDelta)^{2}.
\]

To bound the eigenvalues of $\mbDeltat + \mbDeltat^{T}$, we make use of a diagonalization of $\mbDeltat$.  Given that $\mbV$ is orthogonal, it has unit-modulus eigenvalues and can be diagonalized by a unitary matrix $\mbU$.  From the relations $\mbDelta = \mbV - \mbI$ and $\rho(\mbDelta) < 1/\kappa(\mbQ)$, we see that $\mbDelta$ can be diagonalized as $\mbDelta = \mbU \mbPsi \mbU^{H}$, where the eigenvalues $\psi_{n}$ of $\mbDelta$ lie on the highlighted arc in Fig.~\ref{fig:eigDelta}.  It follows that $\mbDeltat = \mbUt \mbPsi \mbUt^{-1}$ with $\mbUt = \mbLambda^{-1/2} \mbU$.  

\begin{figure}[ht]
\centering
\psfrag{1}[][]{$1$}
\psfrag{-1}[][]{$-1$}
\psfrag{rho}[][]{$\rho(\mbDelta)$}
\includegraphics[width=0.55\columnwidth]{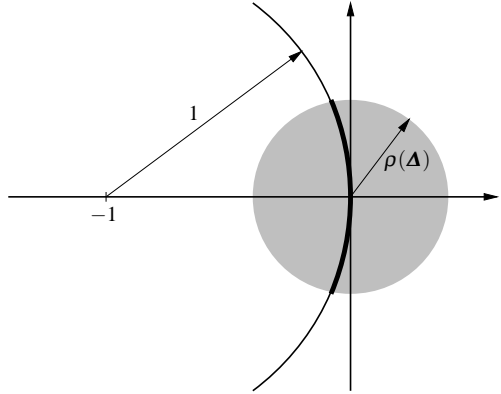}
\caption{The dark segment of the arc indicates the set of possible locations in the complex plane for the eigenvalues of $\mbDelta$ and $\mbDeltat$.}
\label{fig:eigDelta}
\end{figure}

We now invoke a theorem from \cite{hornjohnson1994}, which states that for any eigenvalue of $\mbDeltat + \mbDeltat^{T}$, there exists an eigenvalue of $\mbDeltat$ such that $\abs{\lambda(\mbDeltat + \mbDeltat^{T}) - \lambda(\mbDeltat)} \leq \norm[2]{\mbUt^{-1} \mbDeltat^{T} \mbUt}$.  Expanding the right-hand side of this inequality and using the sub-multiplicative property of spectral norms, we obtain 
\begin{align}
\abs{\lambda(\mbDeltat + \mbDeltat^{T}) - \lambda(\mbDeltat)} 
&\leq \norm[2]{\mbU^{H} \mbLambda^{1/2} \mbLambda^{1/2} \mbDelta^{T} \mbLambda^{-1/2} \mbLambda^{-1/2} \mbU}\notag\\
&\leq \norm[2]{\mbU^{H}} \norm[2]{\mbLambda} \norm[2]{\mbDelta^{T}} \norm[2]{\mbLambda^{-1}} \norm[2]{\mbU}\notag\\
&= \kappa(\mbQ) \rho(\mbDelta).\label{eqn:DeltaDeltaTEig}
\end{align}
The bound in \eqref{eqn:DeltaDeltaTEig} constrains the eigenvalues of $\mbDeltat + \mbDeltat^{T}$ to lie within a Euclidean distance of $\kappa(\mbQ) \rho(\mbDelta)$ from the arc in Fig.~\ref{fig:eigDelta}.  Furthermore, the symmetry of $\mbDeltat + \mbDeltat^{T}$ implies that its eigenvalues are real-valued.  It is clear then that $\lambda_{\max}(\mbDeltat + \mbDeltat^{T}) \leq \kappa(\mbQ) \rho(\mbDelta)$.  From Fig.~\ref{fig:eigDelta} and the assumption that $\kappa(\mbQ) \rho(\mbDelta) < 1$, it can also be seen that $\lambda_{\min}(\mbDeltat + \mbDeltat^{T})$ is minimized by setting $\lambda(\mbDeltat) = 0$ in \eqref{eqn:DeltaDeltaTEig} since all other choices for $\lambda(\mbDeltat)$ would yield more positive values for $\lambda_{\min}(\mbDeltat + \mbDeltat^{T})$.  Substituting the resulting bound $\lambda_{\min}(\mbDeltat + \mbDeltat^{T}) \geq -\kappa(\mbQ) \rho(\mbDelta)$ into \eqref{eqn:QtildeEigMin} completes the proof.
\qed\end{proof}

Theorem \ref{thm:diagRelaxNAA} can now be proved straightforwardly using previous results.

\begin{proof}[Theorem \ref{thm:diagRelaxNAA}]
%{\em Proof of Theorem \ref{thm:diagRelaxNAA}}.
As in the proof of Theorem \ref{thm:diagRelaxDiagDom}, we use Lemma \ref{lem:maxSumSmallestLBD0} to show that $K_{d} \leq \overline{K}_{\na}$, this time choosing $\mbD_{0} = \mbLambda$.  Combining Lemma \ref{lem:maxSumSmallestLBD0} with Lemma \ref{lem:diagRelaxNAA} then yields $E_{d}(K) \geq (1 - \kappa(\mbQ) \rho(\mbDelta)) S_{K} (\{ \lambda_{n}(\mbQ) c_{n}^{2} \})$, which implies that $K_{d} \leq \overline{K}_{\na}$ in light of \eqref{eqn:KoverbarNAA}.

To prove that $K^{\ast} \geq \underline{K}_{\na}$, we proceed as in the proof of Theorem \ref{thm:diagRelaxDiagDom} by fixing a specific subset $Z_{\na}(K)$ corresponding to the $K$ smallest $\lambda_{n}(\mbQ) c_{n}^{2}$.  This yields 
\begin{align*}
E_{0}(K) &\leq \mbc_{Z_{\na}(K)}^{T} \mbQ_{Z_{\na}(K)Z_{\na}(K)} \mbc_{Z_{\na}(K)}\\
&= \left( \mbLambda^{1/2} \begin{bmatrix} \mbc_{Z_{\na}(K)} \\ \0 \end{bmatrix} \right)^{T} \mbLambda^{-1/2} \mbQ \mbLambda^{-1/2} \left( \mbLambda^{1/2} \begin{bmatrix} \mbc_{Z_{\na}(K)} \\ \0 \end{bmatrix} \right)\\
&\leq \lambda_{\max}(\mbLambda^{-1/2} \mbQ \mbLambda^{-1/2}) S_{K}(\{\lambda_{n}(\mbQ) c_{n}^{2}\})\\
&\leq ( 1 + \kappa(\mbQ) (\rho(\mbDelta) + \rho(\mbDelta)^{2}) ) S_{K}(\{\lambda_{n}(\mbQ) c_{n}^{2}\}).
\end{align*}
In the second line above, the quadratic form has been rewritten in terms of the full matrix $\mbQ$ and then rescaled.  The last two lines result from the definition of $Z_{\na}(K)$ and Lemma \ref{lem:diagRelaxNAA}.  Combining the last inequality with \eqref{eqn:KunderbarNAA} yields $K^{\ast} \geq \underline{K}_{\na}$ as desired.
The proof of the bound on $\overline{K}_{\na} / \underline{K}_{\na}$ is similar to that in Theorem \ref{thm:diagRelaxEig}.
\qed\end{proof}

\section{Numerical evaluation}
\label{sec:numEx}

In this section, numerical results are presented to illustrate the performance of the two relaxations discussed in Sect.~\ref{sec:linRelax} and Sect.~\ref{sec:diagRelax}.  In Sect.~\ref{subsec:numExBounds}, the relaxations are compared on the basis of their approximation ratios under different conditions.  In Sect.~\ref{subsec:numExBB}, the relaxations are incorporated in a branch-and-bound algorithm to gauge their effectiveness in reducing the complexity of solving problem \eqref{eqn:sparseProb}.  

\subsection{Approximation ratios}
\label{subsec:numExBounds}

Randomly generated instances of problem \eqref{eqn:sparseProb} are used in this section to evaluate the approximation quality of the two relaxations.  While it was seen in Sect.~\ref{subsec:diagRelaxWorstCase} that neither relaxation dominates the other over all possible instances, the present comparison using random instances indicates that diagonal relaxations yield significantly stronger bounds in many situations, including but not limited to those analyzed in Sect.~\ref{subsec:diagRelaxEig}--\ref{subsec:diagRelaxNAA}.

In these experiments, the problem dimension $N$ is varied between $10$ and $150$ and the parameter $\gamma$ is normalized to $1$ throughout.  The continuous relaxation of each instance, and more specifically the dual problem \eqref{eqn:linRelaxDual}, is solved using the MATLAB function {\tt fmincon}.  A customized solver described in \cite[Sect.~3.5]{wei2011} is used for the diagonal relaxation; a general-purpose semidefinite optimization solver such as SDPT3 \cite{sdpt3} or SeDuMi \cite{sedumi1999} could also be used.  In addition, a feasible solution is obtained for each instance using the backward greedy selection method in \cite{wei2012a}.  The cost of this feasible solution is used as a substitute for the true optimal cost, which is difficult to compute given the large number of instances.  Numerical experience in \cite{wei2012a} however suggests that backward greedy selection is often optimal.  The approximation quality of each relaxation is measured by the ratio of the optimal cost of the relaxation to the cost of the feasible solution.  These ratios are denoted $R_{c}$ and $R_{d}$ for continuous and diagonal relaxations respectively; they are lower bounds on the true approximation ratios.  Note that we are returning to the original definition of approximation ratio in terms of the number of non-zero components and not the number of zero-valued components as in Sect.~\ref{subsec:diagRelaxEig}--\ref{subsec:diagRelaxNAA}.

In the first three experiments, the eigenvector matrix $\mbV$ of $\mbQ$ is chosen uniformly from the set of $N\times N$ orthogonal matrices (as assumed in Theorem \ref{thm:diagRelaxEigRand}).  The eigenvalues are drawn from different power-law distributions and then rescaled to match a specified condition number $\kappa(\mbQ)$ chosen from the values $\sqrt{N}$, $N$, $10N$, and $100N$.  Once $\mbQ$ is fixed, each component of the ellipsoid center $\mbc$ is drawn uniformly from the interval $[-\sqrt{(\mbQ^{-1})_{nn}}, \sqrt{(\mbQ^{-1})_{nn}}]$, in keeping with Assumption \ref{ass:feasTestSingleZero}.

Fig.~\ref{fig:approxRatios}(a) plots the approximation ratios $R_{c}$ and $R_{d}$ as functions of $N$ and $\kappa(\mbQ)$ for an eigenvalue distribution proportional to $1/\lambda$, which corresponds to a uniform distribution for $\log\lambda$.  Each point represents the average of $1000$ instances.  A $1/\lambda$ eigenvalue distribution is unbiased in the sense that it is invariant under matrix inversion (up to a possible overall scaling), an operation that maps the positive definite cone to itself.  The continuous relaxation approximation ratio $R_{c}$ does not vary much with $N$ or $\kappa(\mbQ)$.  In contrast, the diagonal relaxation approximation ratio $R_{d}$ is markedly higher for lower $\kappa(\mbQ)$, in agreement with Theorem \ref{thm:diagRelaxEig} and the geometric intuition in Fig.~\ref{fig:diagRelaxCond}.  Moreover, $R_{d}$ improves with increasing $N$ so that even for $\kappa(\mbQ) = 100N$ the diagonal relaxation outperforms the continuous relaxation for $N \geq 20$, with the difference being substantial at large $N$.  

\begin{figure}[ht]
\centering
\subfigure[]{
\includegraphics[width=0.49\columnwidth]{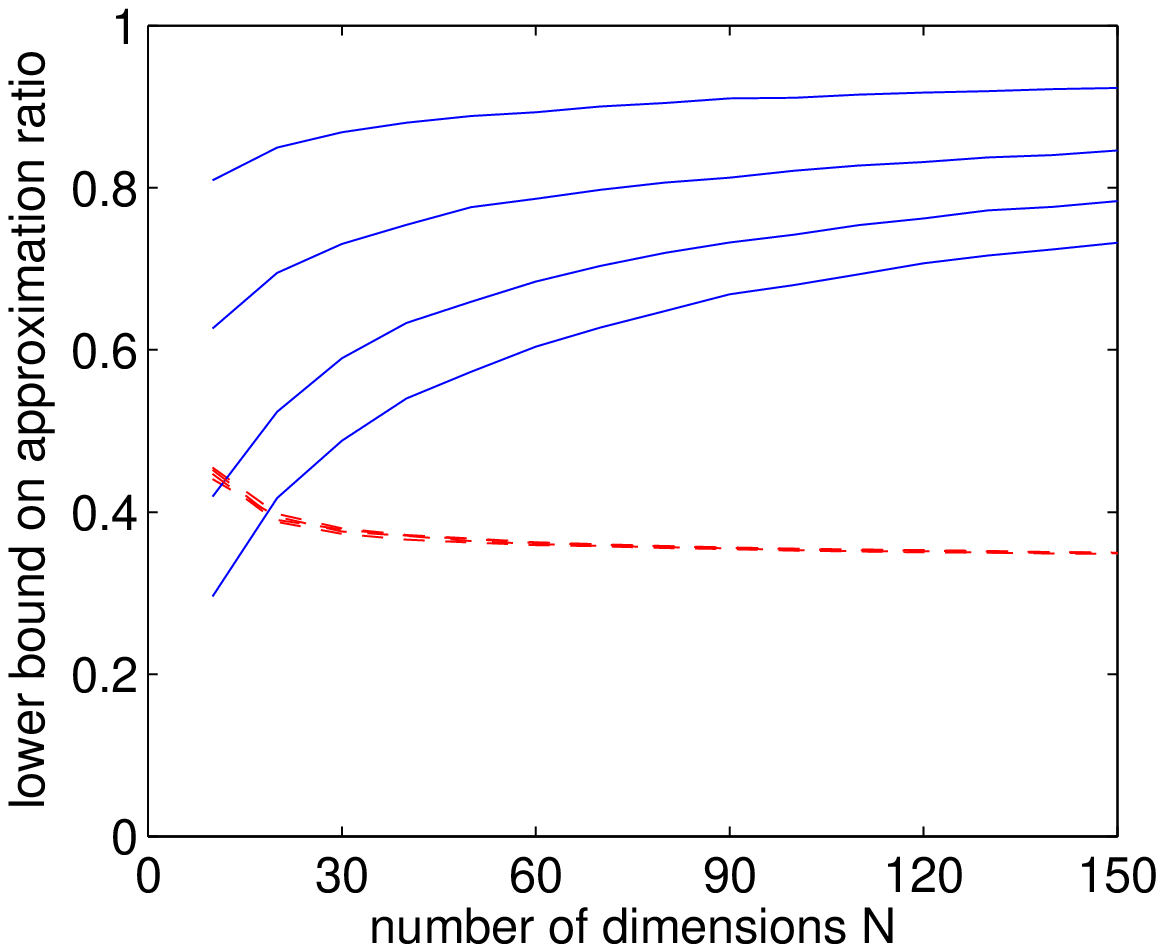}}
\subfigure[]{
\includegraphics[width=0.49\columnwidth]{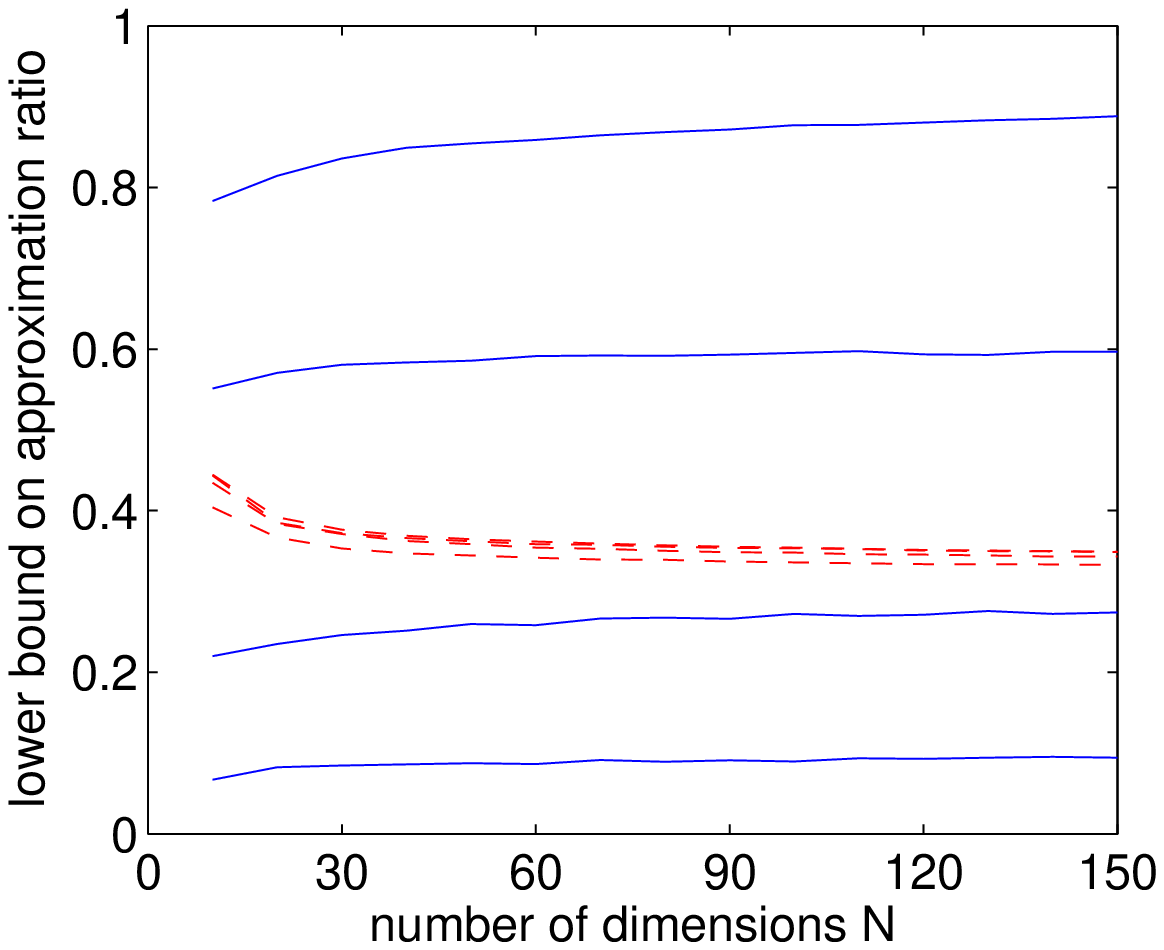}}
\subfigure[]{
\includegraphics[width=0.49\columnwidth]{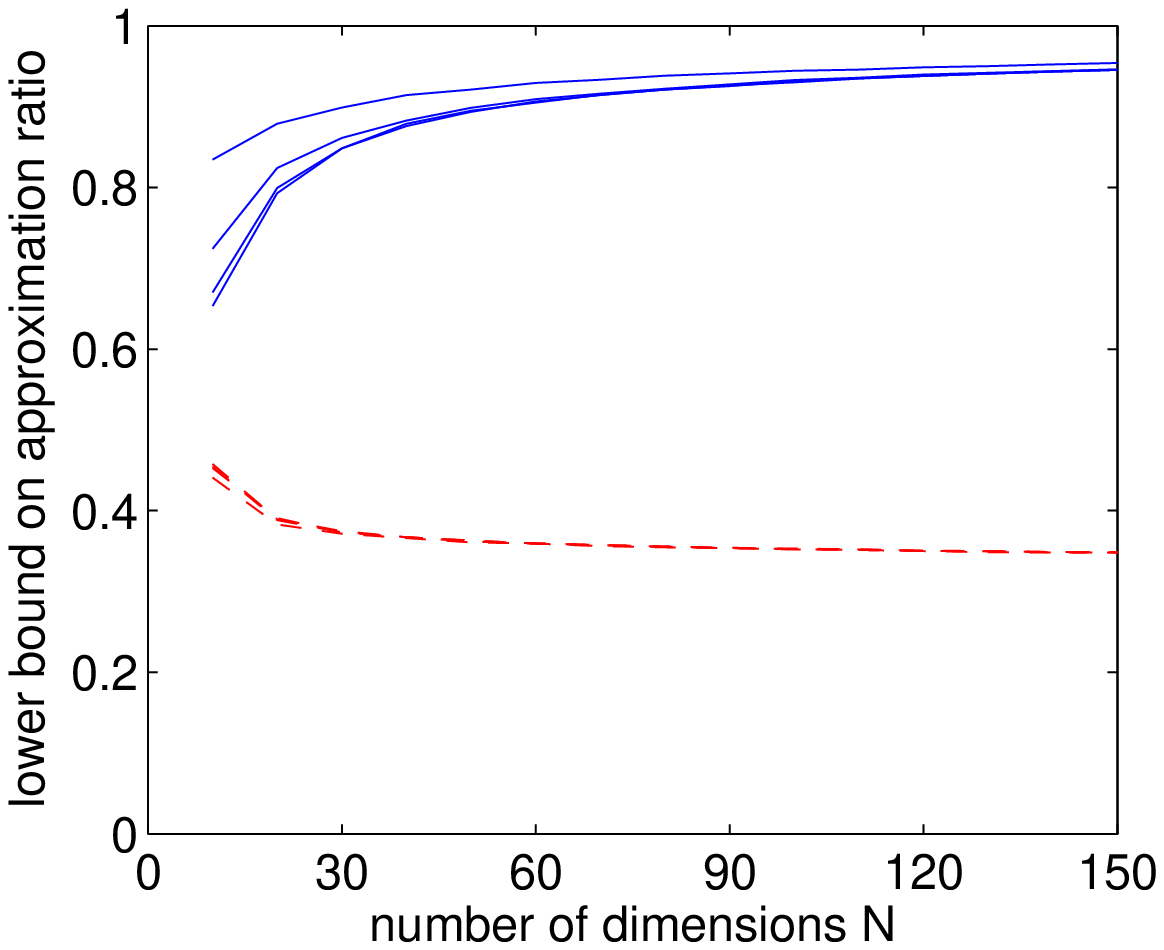}}
\subfigure[]{
\includegraphics[width=0.49\columnwidth]{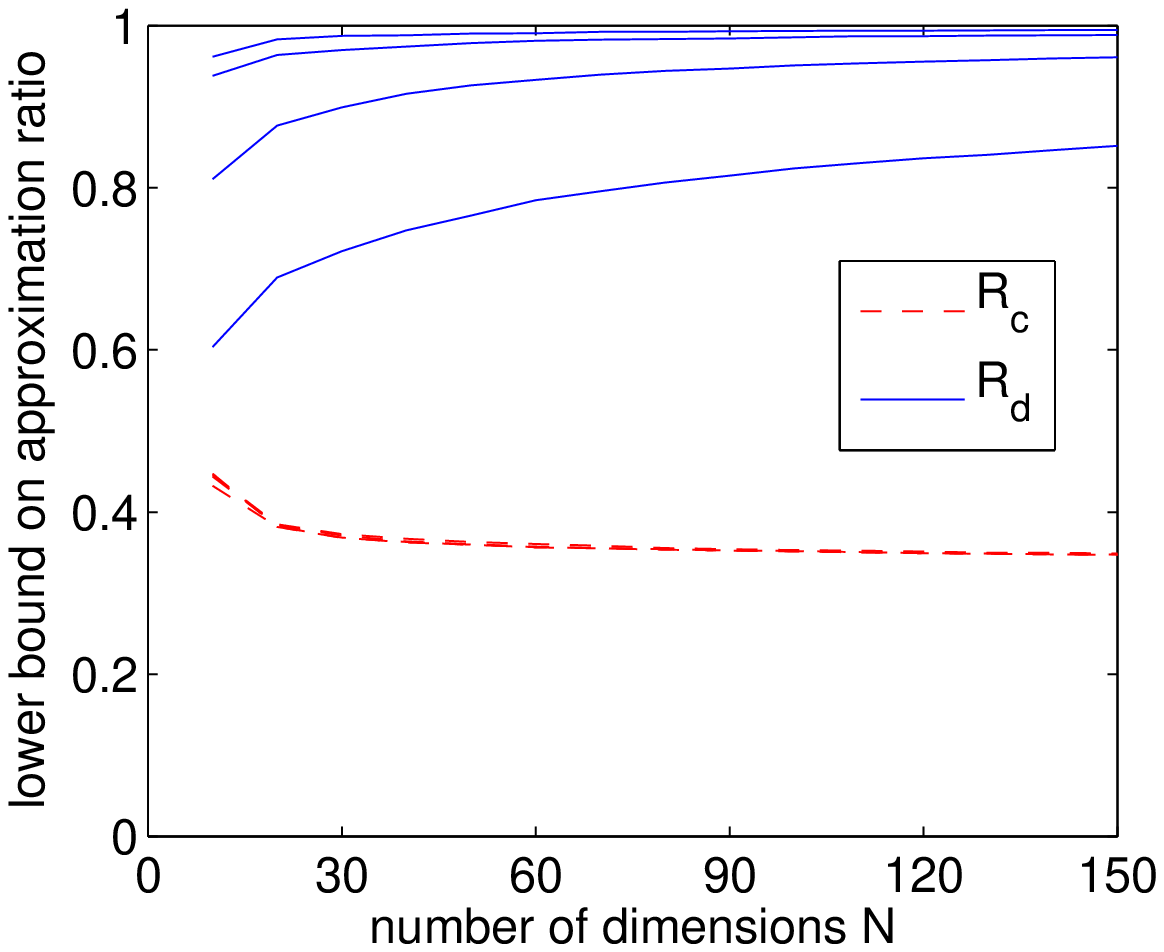}}
\caption{Average approximation ratios $R_c$ and $R_d$ for (a) a $1/\lambda$ eigenvalue distribution, (b) a uniform eigenvalue distribution, (c) a $1/\lambda^{2}$ eigenvalue distribution, and (d) unit diagonal entries and off-diagonal entries drawn uniformly from $[-a, a] / \sqrt{N}$.  In (a)--(c), $\kappa(\mbQ) = \sqrt{N}, N, 10N, 100N$ from top to bottom within each set of curves.  In (d), $a = 0.1, 0.2, 0.5, 0.8$ from top to bottom within each set of curves.} 
\label{fig:approxRatios}
\end{figure}

Figs.~\ref{fig:approxRatios}(b) and \ref{fig:approxRatios}(c) show average approximation ratios for a uniform eigenvalue distribution and a $1/\lambda^{2}$ distribution, the latter corresponding to a uniform distribution for the eigenvalues of $\mbQ^{-1}$.  Compared to a $1/\lambda$ distribution, a $1/\lambda^{2}$ distribution is more heavily weighted toward small values whereas a uniform distribution is less so.  Accordingly, each $R_{d}$ curve in Fig.~\ref{fig:approxRatios}(b) is lower than its counterpart in \ref{fig:approxRatios}(a) while the opposite is true in Fig.~\ref{fig:approxRatios}(c), in agreement with the dependence on the eigenvalue distribution in Theorem \ref{thm:diagRelaxEigRand}.  The effect of the condition number on $R_{d}$ is also more pronounced under a uniform eigenvalue distribution and less so under a $1/\lambda^{2}$ distribution.  The behavior of $R_{c}$ on the other hand is largely unchanged from Fig.~\ref{fig:approxRatios}(a).

In a fourth experiment, the diagonal entries of $\mbQ$ are normalized to $1$ while the off-diagonal entries are drawn uniformly from the interval $[-a, a] / \sqrt{N}$, where $a = 0.1, 0.2, 0.5, 0.8$.  With high probability, such matrices are diagonally dominant in the sense of \eqref{eqn:QdiagDom} for $a = 0.1, 0.2$, and are not positive definite for $a > 0.85$.  The vector $\mbc$ is generated as before based on the diagonal entries of $\mbQ^{-1}$.  %It suffices to consider positive values of $\rho$ since it can be shown that $\rho$ and $-\rho$ are equivalent in terms of the cardinality cost function.  Note also that $\mbQ$ is diagonally dominant in the sense of \eqref{eqn:QdiagDom} for $\rho \leq 1/3$.  
The average approximation ratios are shown in Fig.~\ref{fig:approxRatios}(d).  Similar to the condition number in Figs.~\ref{fig:approxRatios}(a)--(c), the parameter $a$ does not appear to have much effect on $R_{c}$.  For the diagonal relaxation, while Theorem \ref{thm:diagRelaxDiagDom} predicts a close approximation for $a = 0.1, 0.2$, the performance is still relatively good for $a = 0.8$.

The results in Fig.~\ref{fig:approxRatios} demonstrate that better bounds are achieved in many instances with diagonal relaxations than with continuous relaxations.  Furthermore, this can be true even when the condition number $\kappa(\mbQ)$ or the off-diagonal amplitude $a$ is high, whereas the analysis in Sect.~\ref{subsec:diagRelaxEig}--\ref{subsec:diagRelaxNAA} is more conservative.

\subsection{Branch-and-bound complexity}
\label{subsec:numExBB}

Next we consider the effect of the two relaxations on the complexity of a branch-and-bound solution to \eqref{eqn:sparseProb}.  For this purpose, the relaxations are incorporated into a basic MATLAB implementation of branch-and-bound, referred to as BB.  This algorithm is also compared to the mixed-integer programming solver CPLEX 12.4 \cite{cplex2012} as a point of reference.  %Nevertheless, as a preliminary evaluation, the algorithm variant employing diagonal relaxations is also compared to Bonmin \cite{}, a state-of-the-art open-source solver for mixed-integer nonlinear programs, and encouraging results are reported.  
The comparisons show that diagonal relaxations can significantly increase the efficiency of branch-and-bound.  It is also seen that a more specialized solver can outperform a sophisticated general-purpose solver in solving \eqref{eqn:sparseProb}.

Algorithm BB is based on the mixed integer formulation \eqref{eqn:MIP} and is summarized below.  Full details can be found in \cite{wei2012}.  The branching rule is to select the variable for which the margin in condition \eqref{eqn:feasTestSingleZero} is minimal.  This rule is similar to the maximum absolute value rule in \cite{bertsimas2009,bienstock1996} in that the $i_{n} = 0$ subproblem is more likely to be severely constrained.  The next node is chosen according to the ``best node'' rule, i.e., a node with a minimal lower bound.  Feasible solutions are generated by running the backward selection heuristic at every node.  To obtain lower bounds, condition \eqref{eqn:feasTestSingleZero} is checked at every node and bounds are updated as appropriate.  Variable elimination as described in Sect.~\ref{subsec:prelimElim} is employed to reduce subproblem dimensions.  For stronger lower bounds, either continuous or diagonal relaxations are solved, corresponding to two algorithm variants BB-C and BB-D.  Relaxations are solved only after constraining a variable to zero ($i_{n} = 0$ branch) and when the subproblem dimension is at least $20$.  In other cases, the increased computation does not seem to be justified by the improvement in bounds.

For CPLEX, the split-variable mixed integer formulation corresponding to \eqref{eqn:bn+-} and \eqref{eqn:linRelaxPrimal} is passed to the CPLEX MEX executable through the provided MATLAB interface.  Because of the relative inefficiency of CPLEX as seen below, BB is run first and the optimal solution is used to initialize CPLEX.  Given this initialization, CPLEX is instructed to emphasize optimality rather than feasibility, while all other options are set to their default values.  Preliminary experimentation with changing solver parameters did not yield any gains.  The experiments are run on a $2.4$ GHz quad-core Linux computer with $8$ GB of memory.  BB is generally not observed to use more than one core at a time; CPLEX however is able to continuously exploit all four cores.

Problem instances are generated randomly from the same four classes and in the same manner as in Sect.~\ref{subsec:numExBounds}, thus satisfying Assumption \ref{ass:feasTestSingleZero} in particular.  Table \ref{tbl:compareEig} shows the solution times and numbers of nodes for the first three classes in which the eigenvalues of $\mbQ$ are drawn from different distributions.  Each entry represents the average over $100$ instances.  For certain instance classes and solvers, the high computational complexity does not permit an accurate evaluation.  In these cases, the solution time is estimated by extrapolating from lower values of $N$; such estimates are marked by parentheses.

\begin{table}[ht]
\caption{Average computational complexity for different eigenvalue distributions.  Times in parentheses represent extrapolated values.} 
\label{tbl:compareEig}
\begin{center}
%\vskip -3mm
\small
\begin{tabular}{c|c|c||c|c|c||c|c|c}
%\hline
eig.~dist. & $\kappa(\mbQ)$ & $N$ & \multicolumn{3}{c||}{time [s]} & \multicolumn{3}{c}{number of nodes} \\
\hline
& & & BB-C & BB-D & CPLEX & BB-C & BB-D & CPLEX \\
\hline
$1/\lambda$ & $N$ & $40$ & $1.24$ & $0.70$ & $18.38$ & $810$ & $599$ & $5979$ \\
%\hline
& & $70$ & $662$ & $75$ & $2146$ & $2.60\times10^{4}$ & $0.69\times10^{4}$ & $2.72\times10^{5}$ \\
%\hline
& & $100$ & $(4\times10^{5})$ & $1.09\times10^{4}$ & $(2\times10^{5})$ & & $7.40\times10^{4}$ & \\
%\hline
& $100N$ & $40$ & $0.84$ & $0.67$ & $(2\times10^{5})$ & $628$ & $611$ & \\
%\hline
& & $70$ & $334$ & $213$ & & $1.85\times10^{4}$ & $1.46\times10^{4}$ & \\
%\hline
& & $100$ & $(1\times10^{5})$ & $(5\times10^{4})$ & & & & \\
%\hline
\hline
uniform & $N$ & $40$ & $1.09$ & $0.72$ & $15.28$ & $689$ & $616$ & $5500$ \\
%\hline
& & $70$ & $261$ & $98$ & $1159$ & $1.77\times10^{4}$ & $1.01\times10^{4}$ & $2.03\times10^{5}$ \\
%\hline
& & $100$ & $(7\times10^{4})$ & $1.43\times10^{4}$ & $(7\times10^{4})$ & & $1.36\times10^{5}$ & \\
%\hline
& $100N$ & $40$ & $0.18$ & $0.19$ & $(3\times10^{4})$ & $189$ & $189$ & \\
%\hline
& & $70$ & $3.64$ & $2.95$ & & $1.69\times10^{3}$ & $1.69\times10^{3}$ & \\
%\hline
& & $100$ & $98.6$ & $77.1$ & & $9.71\times10^{3}$ & $9.95\times10^{4}$ & \\
%\hline
\hline
$1/\lambda^{2}$ & $N$ & $40$ & $1.93$ & $0.51$ & $23.65$ & $1111$ & $438$ & $6929$ \\
%\hline
& & $70$ & $2949$ & $12$ & $3139$ & $4.72\times10^{4}$ & $0.19\times10^{4}$ & $3.44\times10^{5}$ \\
%\hline
& & $100$ & $(6\times10^{6})$ & $633$ & $(4\times10^{5})$ & & $1.40\times10^{4}$ & \\
%\hline
& $100N$ & $40$ & $1.12$ & $0.40$ & $(3\times10^{5})$ & $742$ & $328$ & \\
%\hline
& & $70$ & $1756$ & $14$ & & $4.19\times10^{4}$ & $0.23\times10^{4}$ & \\
%\hline
& & $100$ & $(1\times10^{6})$ & $848$ & & & $1.60\times10^{4}$ & \\
%\hline
\end{tabular}
\end{center}
%\vskip -7mm
\end{table}

Considering first the comparison between BB-C and BB-D, it is clear from Table \ref{tbl:compareEig} that diagonal relaxations can significantly decrease complexity.  The gains generally increase with the dimension $N$ and can reach several orders of magnitude for the $1/\lambda^{2}$ eigenvalue distribution, which as seen in Sect.~\ref{subsec:numExBounds} is most favorable toward diagonal relaxations.  Even for a uniform distribution and $\kappa(\mbQ) = 100N$, BB-D is slightly more efficient than BB-C, in apparent contradiction with the comparison in Fig.~\ref{fig:approxRatios}(b).  This can be explained by noting that Fig.~\ref{fig:approxRatios}(b) represents the average approximation ratios for the root node whereas subproblems may have more non-uniform eigenvalue distributions and lower condition numbers.  It is also interesting that instances in this class appear to be the easiest to solve.

The comparison with CPLEX in Table \ref{tbl:compareEig} shows the value of a more specialized algorithm for solving \eqref{eqn:sparseProb}, as has been observed by others \cite{bertsimas2009,gao2011}.  This is in spite of the fact that CPLEX is run as a compiled executable with full multicore capabilities.  Indeed, the advantage extends to the BB-C variant at low $N$, although the margin decreases at higher $N$.  Note also that CPLEX has difficulty with the more poorly-conditioned instances.  Given CPLEX's use of techniques beyond pure branch-and-bound, it is difficult to identify precisely the reasons for its relative inefficiency.  One factor is the poor performance of the heuristic used by CPLEX relative to the backward selection heuristic in BB.  For this reason, CPLEX is initialized with the BB solution in the experiments.  As for lower bounds, it is likely that checking condition \eqref{eqn:feasTestSingleZero} confers significant benefits because of the ability to eliminate many infeasible subproblems and improve bounds incrementally with minimal computation, and also because of the subsequent reduction in dimension.  Another difference is the frequency at which relaxations are solved since in BB, some effort is made to avoid solving unprofitable relaxations.

Table \ref{tbl:compareDiagDom} shows a complexity comparison for $\mbQ$ matrices with unit diagonal entries and uniformly distributed off-diagonal entries, corresponding to Fig.~\ref{fig:approxRatios}(d) in Sect.~\ref{subsec:numExBounds}.  The difference between BB-C and BB-D in this case is as dramatic as it is for the $1/\lambda^{2}$ eigenvalue distribution in Table \ref{tbl:compareEig}.  The performance of CPLEX is similar to its performance in Table \ref{tbl:compareEig} for $\kappa(\mbQ) = N$.  It is clear that BB-D remains the best option.

\begin{table}[ht]
\caption{Average computational complexity for different off-diagonal amplitudes $a$.  Times in parentheses represent extrapolated values.} 
\label{tbl:compareDiagDom}
\begin{center}
%\vskip -3mm
\small
\begin{tabular}{c|c||c|c|c||c|c|c}
%\hline
$a$ & $N$ & \multicolumn{3}{c||}{time [s]} & \multicolumn{3}{c}{number of nodes} \\
\hline
& & BB-C & BB-D & CPLEX & BB-C & BB-D & CPLEX \\
\hline
$0.2$ & $40$ & $1.66$ & $0.13$ & $26.87$ & $1128$ & $93$ & $8698$ \\
%\hline
& $70$ & $2941$ & $1.1$ & $4107$ & $6.76\times10^{4}$ & $151$ & $4.85\times10^{5}$ \\
%\hline
& $100$ & $(7\times10^{6})$ & $2.6$ & $(9\times10^{5})$ & & $187$ & \\
%\hline
$0.8$ & $40$ & $1.56$ & $0.76$ & $24.04$ & $849$ & $543$ & $7896$ \\
%\hline
& $70$ & $577$ & $50$ & $2853$ & $3.21\times10^{4}$ & $0.57\times10^{4}$ & $3.86\times10^{5}$ \\
%\hline
& $100$ & $(4\times10^{5})$ & $4.86\times10^{3}$ & $(4\times10^{5})$ & & $7.51\times10^{4}$ & \\
%\hline
\end{tabular}
\end{center}
%\vskip -7mm
\end{table}

\section{Conclusion and future work}
\label{sec:concl}

Two relaxations of a quadratically-constrained cardinality minimization problem \eqref{eqn:sparseProb} were investigated, the first being the continuous relaxation of a mixed integer formulation, the second an optimized diagonal relaxation based on a simple special case of the problem.  An absolute upper bound on the optimal cost of the continuous relaxation suggests that it yields relatively weak approximations.  In computational experiments, diagonal relaxations were seen to result in stronger bounds and significantly reduced complexity in solving \eqref{eqn:sparseProb} via branch-and-bound.  Substantial gains were also observed relative to the general-purpose solver CPLEX.  To support these numerical results, this paper analyzed the approximation properties of diagonal relaxations, providing general insight and establishing guarantees in terms of the eigenvalues of the matrix $\mbQ$ and in the diagonally dominant and nearly coordinate-aligned cases.

Given the interest in generalizations of \eqref{eqn:sparseProb} in portfolio optimization, it is hoped that the analysis in this paper could be extended to these more general formulations and to other relaxations such as the perspective relaxation \cite{frangioni2006,gunluk2010,zheng2012}.  In addition, the positive experience with diagonal relaxations motivates further exploration of relaxations based on other efficiently solvable special cases, for example those in \cite{das2008}.

\begin{comment}
\section{Introduction}
\label{intro}
Your text comes here. Separate text sections with
\section{Section title}
\label{sec:1}
Text with citations \cite{RefB} and \cite{RefJ}.
\subsection{Subsection title}
\label{sec:2}
as required. Don't forget to give each section
and subsection a unique label (see Sect.~\ref{sec:1}).
\paragraph{Paragraph headings} Use paragraph headings as needed.
\begin{equation}
a^2+b^2=c^2
\end{equation}

\begin{figure}
  \includegraphics{example.eps}
\caption{Please write your figure caption here}
\label{fig:1}       % Give a unique label
\end{figure}
\begin{figure*}
  \includegraphics[width=0.75\textwidth]{example.eps}
\caption{Please write your figure caption here}
\label{fig:2}       % Give a unique label
\end{figure*}
\begin{table}
\caption{Please write your table caption here}
\label{tab:1}       % Give a unique label
\begin{tabular}{lll}
\hline\noalign{\smallskip}
first & second & third  \\
\noalign{\smallskip}\hline\noalign{\smallskip}
number & number & number \\
number & number & number \\
\noalign{\smallskip}\hline
\end{tabular}
\end{table}
\end{comment}

\begin{acknowledgements}
The author thanks Pablo A.~Parrilo, Alan V.~Oppenheim and Vivek K.~Goyal for helpful discussions that shaped this work.
\end{acknowledgements}

\bibliographystyle{spmpsci}      % mathematics and physical sciences
\bibliography{IEEEabrv,sparseQuad,optimization}   % name your BibTeX data base

\begin{thebibliography}{10}
\providecommand{\url}[1]{{#1}}
\providecommand{\urlprefix}{URL }
\expandafter\ifx\csname urlstyle\endcsname\relax
  \providecommand{\doi}[1]{DOI~\discretionary{}{}{}#1}\else
  \providecommand{\doi}{DOI~\discretionary{}{}{}\begingroup
  \urlstyle{rm}\Url}\fi

\bibitem{bertsekas1999}
Bertsekas, D.P.: Nonlinear Programming.
\newblock Athena Scientific, Belmont, MA (1999)

\bibitem{bertsimas2009}
Bertsimas, D., Shioda, R.: Algorithm for cardinality-constrained quadratic
  optimization.
\newblock Comput. Optim. Appl. \textbf{43}, 1--22 (2009)

\bibitem{bt1997}
Bertsimas, D., Tsitsiklis, J.N.: Introduction to Linear Optimization.
\newblock Athena Scientific, Nashua, NH (1997)

\bibitem{bienstock1996}
Bienstock, D.: Computational study of a family of mixed-integer quadratic
  programming problems.
\newblock Math. Program. \textbf{74}(2), 121--140 (1996)

\bibitem{bonami2009}
Bonami, P., Lejeune, M.A.: An exact solution approach for portfolio
  optimization problems under stochastic and integer constraints.
\newblock Oper. Res. \textbf{57}(3), 650--670 (2009)

\bibitem{bv2004}
Boyd, S., Vandenberghe, L.: Convex Optimization.
\newblock Cambridge University Press, Cambridge, UK (2004)

\bibitem{ceria1999}
Ceria, S., Soares, J.: Convex programming for disjunctive convex optimization.
\newblock Math. Program. \textbf{86}(3), 595--614 (1999)

\bibitem{chernoff1952}
Chernoff, H.: A measure of asymptotic efficiency for tests of a hypothesis
  based on the sum of observations.
\newblock Ann. Math. Stat. \textbf{23}(4), 493--507 (1952)

\bibitem{couvreur2000}
Couvreur, C., Bresler, Y.: On the optimality of the backward greedy algorithm
  for the subset selection problem.
\newblock SIAM J. Matrix Anal. Appl. \textbf{21}(3), 797--808 (2000)

\bibitem{cui2012}
Cui, X., Zheng, X., Zhu, S., Sun, X.: Convex relaxations and {MIQCQP}
  reformulations for a class of cardinality-constrained portfolio selection
  problems.
\newblock J. Glob. Optim. pp. 1--15 (2012)

\bibitem{das2008}
Das, A., Kempe, D.: Algorithms for subset selection in linear regression.
\newblock In: 40th ACM Symposium on Theory of Computing (STOC), pp. 45--54.
  Victoria, Canada (2008)

\bibitem{frangioni2006}
Frangioni, A., Gentile, C.: Perspective cuts for a class of convex 0-1 mixed
  integer programs.
\newblock Math. Program., Ser. A \textbf{106}(2), 225--236 (2006)

\bibitem{frangioni2007}
Frangioni, A., Gentile, C.: {SDP} diagonalizations and perspective cuts for a
  class of nonseparable {MIQP}.
\newblock Oper. Res. Lett. \textbf{35}(2), 181--185 (2007)

\bibitem{gao2011}
Gao, J., Li, D.: Cardinality constrained linear-quadratic optimal control.
\newblock {IEEE} Trans. Autom. Control \textbf{56}(8), 1936--1941 (2011)

\bibitem{gao2012}
Gao, J., Li, D.: A polynomial case of the cardinality-constrained quadratic
  optimization problem.
\newblock J. Glob. Optim. pp. 1--15 (2012)

\bibitem{gunluk2010}
G\"{u}nl\"{u}k, O., Linderoth, J.: Perspective reformulations of mixed integer
  nonlinear programs with indicator variables.
\newblock Math. Program., Ser. B \textbf{124}(1-2), 183--205 (2010)

\bibitem{hornjohnson1994}
Horn, R.A., Johnson, C.R.: Topics in Matrix Analysis.
\newblock Cambridge University Press, Cambridge, UK (1994)

\bibitem{cplex2012}
IBM ILOG: IBM ILOG CPLEX 12.4 User's Manual (2012)

\bibitem{lemarechaloustry1999}
Lemar\'{e}chal, C., Oustry, F.: Semidefinite relaxations and {Lagrangian}
  duality with application to combinatorial optimization.
\newblock Tech. Rep. RR-3710, INRIA (1999)

\bibitem{lin2012}
Lin, F., Fardad, M., Jovanovic, M.R.: Design of optimal sparse feedback gains
  via the alternating direction method of multipliers (2012).
\newblock ArXiv preprint, \url{http://arxiv.org/abs/1111.6188}

\bibitem{miller2002}
Miller, A.J.: Subset selection in regression, 2 edn.
\newblock Chapman \& Hall/CRC, Boca Raton, FL (2002)

\bibitem{shaw2008}
Shaw, D.X., Liu, S., Kopman, L.: Lagrangian relaxation procedure for
  cardinality-constrained portfolio optimization.
\newblock Optim. Method. Softw. \textbf{23}(3), 411--420 (2008)

\bibitem{sedumi1999}
Sturm, J.F.: Using {SeDuMi} 1.02, a {Matlab} toolbox for optimization over
  symmetric cones.
\newblock Optim. Method. Softw. \textbf{11}, 625--653 (1999)

\bibitem{sdpt3}
Toh, K.C., Todd, M.J., T\"{u}t\"{u}nc\"{u}, R.H.: {SDPT3} --- a {MATLAB}
  software package for semidefinite programming.
\newblock Optim. Method. Softw. \textbf{11}, 545--581 (1999).
\newblock Latest version available at
  \url{http://www.math.nus.edu.sg/~mattohkc/sdpt3.html}

\bibitem{vielma2008}
Vielma, J.P., Ahmed, S., Nemhauser, G.L.: A lifted linear programming
  branch-and-bound algorithm for mixed-integer conic quadratic programs.
\newblock INFORMS J. Comput. \textbf{20}(3), 438--450 (2008)

\bibitem{wei2011}
Wei, D.: Design of discrete-time filters for efficient implementation.
\newblock Ph.D. thesis, Massachusetts Institute of Technology, Cambridge, MA
  (2011)

\bibitem{wei2012}
Wei, D., Oppenheim, A.V.: A branch-and-bound algorithm for
  quadratically-constrained sparse filter design.
\newblock {IEEE} Trans. Signal Process.  (to appear)

\bibitem{wei2012a}
Wei, D., Sestok, C.K., Oppenheim, A.V.: Sparse filter design under a quadratic
  constraint: Low-complexity algorithms.
\newblock {IEEE} Trans. Signal Process.  (to appear)

\bibitem{zheng2012}
Zheng, X., Sun, X., Li, D.: Improving the performance of {MIQP} solvers for
  quadratic programs with cardinality and minimum threshold constraints: A
  semidefinite program approach.
\newblock Preprint,
  \url{http://www.optimization-online.org/DB_FILE/2010/11/2797.pdf}

\end{thebibliography}

%\begin{thebibliography}{}
% etc
%\end{thebibliography}

\end{document}